\documentclass{article}

\usepackage{authblk}
\usepackage{a4wide}
\usepackage{amsthm}
\usepackage{graphicx}
\usepackage{amssymb}
\usepackage{amsmath}
\usepackage{ascmac}
\usepackage{setspace}
\usepackage{float}
\usepackage{algorithm}
\usepackage{algorithmic}
\usepackage{setspace}
\usepackage{tikz-cd}
\usepackage{appendix}
\newtheorem{Theorem}[equation]{Theorem}

\newtheorem{Lemma}[equation]{Lemma}

\theoremstyle{definition}
\newtheorem{Definition}[equation]{Definition}

\theoremstyle{remark}
\newtheorem{Remark}[equation]{Remark}
\numberwithin{equation}{section}

\DeclareMathOperator{\ev}{ev}
\DeclareMathOperator{\id}{id}

\DeclareMathOperator{\ad}{ad}

\DeclareMathOperator{\row}{row}
\DeclareMathOperator{\col}{col}

\newcommand{\ve}{\varepsilon}

\allowdisplaybreaks
\title{Affine super Yangians and non-rectangular $W$-superalgebras}
\author{Mamoru Ueda\thanks{mueda@ualberta.ca}}
\affil{Department of Mathematical and Statistical Sciences, University of Alberta, 11324 89 Ave NW, Edmonton, AB T6G 2J5, Canada}
\date{}
\begin{document}
\maketitle
\begin{abstract}
We construct four edge contractions for the affine super Yangian of type $A$. By using these edge contractions, we give a homomorphism from the affine super Yangian of type $A$ to the universal enveloping algebra of the non-rectangular $W$-superalgebra of type $A$.
\end{abstract}
\section{Introduction}
For a finite dimensional simple Lie algebra $\mathfrak{g}$, Drinfeld (\cite{D1}, \cite{D2}) defined the (finite) Yangian $Y_\hbar(\mathfrak{g})$ as a deformation of the current algebra $\mathfrak{g}\otimes\mathbb{C}[z]$. The finite Yangian of type $A$ has several presentations: RTT presentation, current presentation, parabolic presentation and so on. By using the current presentation, the definition of the Yangian $Y_\hbar(\mathfrak{g})$ can be extended to a Kac-Moody Lie algebra. 
In the case that the $\mathfrak{g}$ is of affine type, the coproduct for the Yangian was given by Boyarchenko-Levendorski\u{\i} \cite{BL} (type $A^{(1)}_1$ case), 
Guay-Nakajima-Wendlandt \cite{GNW} (except of types $A^{(1)}_1$ and $A^{(2)}_2$) and the author \cite{U1} (type $A^{(2)}_2$ case). In type $A$ setting, Guay (\cite{Gu2} and \cite{Gu1}) introduced the 2-parameter affine Yangian associated with $\widehat{\mathfrak{sl}}(n)$, $Y_{\hbar,\ve}(\widehat{\mathfrak{sl}}(n))$, which is the deformation of the universal enveloping algebra of the universal central extension of $\mathfrak{sl}(n)[u^{\pm1},v]$. Guay also gave a non-trivial homomorphism from the affine Yangian $Y_{\hbar,\ve}(\widehat{\mathfrak{sl}}(n))$ to the standard degreewise completion of the universal enveloping algebra of $\widehat{\mathfrak{gl}}(n)$, which is called the evaluation map.

Recently, the Yangians are used for the study of $W$-algebras.
It was shown by Ragoucy-Sorba \cite{RS} that there exist surjective homomorphisms from finite Yangians of type $A$ to finite rectangular $W$-algebras of type $A$. More generally, Brundan and Kleshchev (\cite{BK}) wrote down a finite $W$-algebra of type $A$ as a quotient algebra of a shifted Yangian of type $A$. The shifted Yangian of type $A$ contains finite Yangians of type $A$ as subalgebras. By restricting the homomorphism in \cite{BK} to a finite Yangian of type $A$, we obtain a homomorphism from a finite Yangian of type $A$ to a finite $W$-algebra of type $A$. De Sole-Kac-Valeri \cite{DKV} gave this homomorphism by using the Lax operator.

In the affine setting, by using a geometric realization of the Yangian, Schiffman-Vasserot (\cite{SV}) constructed a surjective homomorphism from the Yangian of $\widehat{\mathfrak{gl}}(1)$ to the universal enveloping algebras of the principal $W$-algebras of type $A$ and have proved the celebrated AGT conjecture (\cite{Ga}, \cite{BFFR}). For rectangular $W$-algebras,
we \cite{U4} constructed a homomorphism from the affine Yangian $Y_{\hbar,\ve}(\widehat{\mathfrak{sl}}(n))$ to the universal enveloping algebra of the rectangular $W$-algebra $\mathcal{W}^k(\mathfrak{gl}(ln),(l^n))$. Kodera and the author \cite{KU} gave another proof to \cite{U4} by using the coproduct and evaluation map for the affine Yangian. For a non-rectangular $W$-algebra, the author (\cite{U7},\cite{U9}) gave a homomorphism from the affine Yangian $Y_{\hbar,\ve}(\widehat{\mathfrak{sl}}(n))$ to the universal enveloping algebra of a non-rectangular $W$-algebra of type $A$, which is an affinization of the homomorphism of De Sole-Kac-Valeri. In \cite{U11} and \cite{U12}, we constructed this homomorphism by using the coproduct and evaluation map for the affine Yangian of type $A$.
We expect that this homomprphism is helpful to resolve Crutzig-Diaconescu-Ma's conjecture \cite{CE}, which notes that an action of an iterated $W$-algebra of type $A$ on the equivariant homology space of the affine Laumon space will be given through an action of an shifted affine Yangian constructed in \cite {FT}. 

As for the Lie superalgebras, Nazarov \cite{Na} introduced the Yangian associated with $\mathfrak{gl}(m|n)$ and Stukopin \cite{S} defined the Yangian of $\mathfrak{sl}(m|n)$. The author \cite{U2} defined the affine super Yangian associated with $\widehat{\mathfrak{sl}}(m|n)$. In \cite{U2}, we constructed the coproduct for the affine super Yangian:
\begin{equation*}
\Delta\colon Y_{\hbar,\ve}(\widehat{\mathfrak{sl}}(m|n))\to Y_{\hbar,\ve}(\widehat{\mathfrak{sl}}(m|n))\widehat{\otimes}Y_{\hbar,\ve}(\widehat{\mathfrak{sl}}(m|n)),
\end{equation*}
where $Y_{\hbar,\ve}(\widehat{\mathfrak{sl}}(m|n))\widehat{\otimes}Y_{\hbar,\ve}(\widehat{\mathfrak{sl}}(m|n))$ is the standard degreewise completion of $\bigotimes^2Y_{\hbar,\ve}(\widehat{\mathfrak{sl}}(m|n))$. In \cite{U2}, we also gave the evaluation map
\begin{equation*}
\ev\colon Y_{\hbar,\ve}(\widehat{\mathfrak{sl}}(m|n))\to\mathcal{U}(\widehat{\mathfrak{gl}}(m|n)),
\end{equation*}
where $\mathcal{U}(\widehat{\mathfrak{gl}}(m|n))$ is the standard degreewise completion of $U(\widehat{\mathfrak{gl}}(m|n))$. 
A relationship between super Yangians and finite $W$-superalgebras was constructed by Briot and Ragoucy \cite{BR} for the rectangular case and by Peng \cite{Pe} for more general case.  In affine setting, Gaberdiel, Li, Peng and Zhang (\cite{GLPZ}) defined the Yangian for the affine Lie superalgebra $\widehat{\mathfrak{gl}}(1|1)$ and obtained a result similar to \cite{SV} in the super setting. As for rectangular $W$-superalgebras, we \cite{U4} gave a homomorphism from the affine super Yangian to the universal enveloping algebra of a $W$-superalgebra of type $A$.

In this article, we construct a homomorphism from the affine super Yangian $Y_{\hbar,\ve}(\widehat{\mathfrak{sl}}(m|n))$ to the universal enveloping algebra of a $W$-superalgebra of type $A$. 
Let us set
\begin{align*}
M&=\displaystyle\sum_{i=1}^lu_i,\qquad u_1\geq u_{2}\geq\cdots\geq u_l\geq u_{l+1}=0,\nonumber\\
N&=\displaystyle\sum_{i=1}^lq_i,\qquad q_1\geq q_{2}\geq\cdots\geq q_l\geq q_{l+1}=0,
\end{align*}
and assume that $u_l+q_l\neq0$ and $M\neq N$. Let us take a nilpotent element $f\in\mathfrak{gl}(M|N)$ of type $(1^{u_1-u_2|q_1-q_2},2^{u_2-u_3|q_2-q_3},\cdots,l^{u_l-u_{l+1}|q_l-q_{l+1}})$.

We use the Miura map for the $W$-superalgebra $\mathcal{W}^k(\mathfrak{gl}(M|N),f)$ given by Nakatsuka \cite{Nak}. The Miura map induces a homomorphism
\begin{equation*}
\widetilde{\mu}\colon \mathcal{U}(\mathcal{W}^{k}(\mathfrak{gl}(M|N),f))\to {\widehat{\bigotimes}}_{1\leq s\leq l}U(\widehat{\mathfrak{gl}}(u_s|q_s)),
\end{equation*}
where ${\widehat{\bigotimes}}_{1\leq s\leq l}U(\widehat{\mathfrak{gl}}(u_s|q_s))$ is the standard degreewise completion of $\bigotimes_{1\leq s\leq l}U(\widehat{\mathfrak{gl}}(u_s|q_s))$. We construct two edge contractions for the affine super Yangian:
\begin{gather*}
\Psi_1^{m_1|n_1,m_1+m_2|n_1+n_2}\colon Y_{\hbar,\ve}(\widehat{\mathfrak{sl}}(m_1|n_1))\to \widetilde{Y}_{\hbar,\ve}(\widehat{\mathfrak{sl}}(m_1+m_2|n_1+n_2)),\\
\Psi_2^{m_2|n_2,m_1+m_2|n_1+n_2}\colon Y_{\hbar,\ve+(m_1-n_1)\hbar}(\widehat{\mathfrak{sl}}(m_2|n_2))\to \widetilde{Y}_{\hbar,\ve}(\widehat{\mathfrak{sl}}(m_1+m_2|n_1+n_2)).
\end{gather*}
\begin{Theorem}\label{A}
Let us assume that $u_s-u_{s+1},q_s-q_{s+1}\geq2$ and $u_s-u_{s+1}+q_s-q_{s+1}\geq5$ and $\ve_s=(k+M-N-u_s+q_s)\hbar$. Then, there exists an algebra homomorphism 
\begin{equation*}
\Phi\colon Y_{\hbar,\ve_s}(\widehat{\mathfrak{sl}}(u_s-u_{s+1}|q_s-q_{s+1}))\to \mathcal{U}(\mathcal{W}^{k}(\mathfrak{gl}(N),f))
\end{equation*} 
determined by
\begin{equation*}
\ev^{\otimes s}\circ(\prod_{a=1}^{s-1}((\Psi_2^{u_{a+1}|q_{a+1},u_a|q_a}\otimes\id)\circ\Delta)\otimes\id^{s-a-1})\circ\Psi_1^{u_s-u_{s+1}|q_s-q_{s+1},u_s|q_s}=\widetilde{\mu}\circ\Phi.
\end{equation*}
\end{Theorem}
Theorem~\ref{A} is helpful for the generalized Gaiotto-Rapcak's triality.
Gaiotto-Rapcak \cite{GR} introduced a vertex algebra called the $Y$-algebra. The $Y$-algebra can be interpreted as a truncation of $\mathcal{W}_{1+\infty}$-algebra (\cite{GG}) whose universal enveloping algebra is isomorphic to the affine Yangian of $\widehat{\mathfrak{gl}}(1)$ up to suitable completions (see \cite{AS}, \cite{T} and \cite{MO}). Creutzig-Linshaw \cite{CR} have proved the triality conjecture in some cases. They gave the isomorphism between affine cosets of some $W$-algebras and those of some $W$-superalgebras.
This result is the generalization of the Feigin-Frenkel duality and the coset realization of principal $W$-algebra. 
Similarly to edge contractions of the affine Yangian (\cite{U10}), we expect that the images of $\Psi_1^{m_1|n_1,m_1+m_2|n_1+n_2}$ and $\Psi_2^{m_2|n_2,m_1+m_2|n_1+n_2}$ are commutative with each other. If we can show this commutativity, we obtain a homomorphism from the affine super Yangian of type $A$ to some cosets of $W$-superalgebras of type $A$ in the similar way to \cite{U10} and \cite{U12}. We will come back to this commutativity in the future work.

\section*{Acknowledgement}
The author expresses his sincere thanks to Thomas Creutzig for the helpful discussion. This work was supported by Japan Society for the Promotion of Science Overseas Research Fellowships,
Grant Number JP2360303.

\section{The affine Lie superalgebra $\widehat{\mathfrak{sl}}(m|n)$}
Let us take integers $m,n\geq 2$ and $m+n\geq 5$. We set
\begin{align*}
I&=\{1,2,\dots,m,-1,-2,\dots,-n\}
\end{align*}
and define the parity on $I$ by
$$p(i)=\begin{cases}
0&\text{ if }i>0,\\
1&\text{ if }i<0.
\end{cases}$$
Sometimes, we identify $I$ with $\mathbb{Z}/(m+n)\mathbb{Z}$ by corresponding $-i\in I$ to $m+i\in\mathbb{Z}/(m+n)\mathbb{Z}$ for $1\leq i\leq n$.

Let us set the Lie superalgebra $\mathfrak{gl}(m|n)=\bigoplus_{i,j\in I}\mathbb{C}E_{i,j}$ with the commutator relation:
\begin{equation*}
[E_{i,j},E_{x,y}]=\delta_{j,x}E_{i,y}-(-1)^{p(E_{i,j})p(E_{x,y})}\delta_{i,y}E_{x,j},
\end{equation*}
where $p(E_{i,j})=p(i)+p(j)$. We also define the affinization of $\mathfrak{gl}(m|n)$:
\begin{equation*}
\widehat{\mathfrak{gl}}(m|n)=\mathfrak{gl}(m|n)\otimes\mathbb{C}[t^{\pm1}]\oplus\mathbb{C}c\oplus\mathbb{C}z
\end{equation*}
whose commutator relations are given by
\begin{align*}
[E_{i,j}t^u, E_{x,y}t^v]
&=\delta_{j,x}E_{i,y}t^{u+v}-(-1)^{p(E_{i,j})p(E_{x,y})}\delta_{i,y}E_{x,j}t^{u+v}\\
&\quad+\delta_{u+v,0}u\delta_{i,y}\delta_{j,x}(-1)^{p(i)}c+\delta_{u+v,0}u\delta_{i,j}\delta_{x,y}(-1)^{p(i)+p(x)}z\\
&\qquad\text{$z$ and $c$ are central elements of }\widehat{\mathfrak{gl}}(m|n).
\end{align*}
We define Lie super subalgebras of $\mathfrak{gl}(m|n)$ and $\widehat{\mathfrak{gl}}(m|n)$:
\begin{gather*}
\mathfrak{sl}(m|n)=\{\sum_{i,j\in I}\limits b_{i,j}E_{i,j}\in\mathfrak{gl}(m|n)\mid\text{str}(\sum_{i,j\in I}\limits b_{i,j}E_{i,j})=\sum_{i\in I}(-1)^{p(i)}b_{i,i}=0\},\\
\widehat{\mathfrak{sl}}(m|n)=\mathfrak{sl}(m|n)\otimes\mathbb{C}[t^{\pm1}]\oplus\mathbb{C}c\subset\widehat{\mathfrak{gl}}(m|n).
\end{gather*}
By Theorem 3.5.1 and Theorem 4.1.1 in \cite{Yamane}, we prepare one presentation of $\widehat{\mathfrak{sl}}(m|n)$. 
We set the Cartan matrix $(a_{i,j})_{i,j\in\mathbb{Z}/(m+n)\mathbb{Z}}$ as
\begin{equation*}
a_{i,j}=\begin{cases}
(-1)^{p(i)}+(-1)^{p(i+1)}&\text{if } i=j,\\
-(-1)^{p(i+1)}&\text{if }j=i+1,\\
-(-1)^{p(i)}&\text{if }j=i-1,\\
0&\text{otherwise}.
\end{cases}
\end{equation*}
\begin{Theorem}\label{presentation}
The Lie superalgebra $\widehat{\mathfrak{sl}}(m|n)$ is isomorphic to the Lie superalgebra generated by
\begin{equation*}
\{h_i,x^\pm_i\mid0\leq i\leq m+n-1\}
\end{equation*}
generated by
\begin{gather*}
[h_{i}, h_{j}] = 0,\\
[x_{i}^{+}, x^-_{j}] = \delta_{i,j} h_{i},\\
[h_{i}, x_{j}^{\pm}] = \pm a_{i,j} x_{j}^{\pm},\\
(\ad x_{i}^{\pm})^{1+|a_{i,j}|} (x_{j}^{\pm})= 0 \text{ if }i \neq j, \\
[x^\pm_{i,0},x^\pm_{i,0}]=0\text{ if }p(i)\neq p(i+1),\\
[[x^\pm_{i-1,0},x^\pm_{i,0}],[x^\pm_{i,0},x^\pm_{i+1,0}]]=0\text{ if }p(i)\neq p(i+1),
\end{gather*}
where $x^\pm_i$ are odd if $p(i)\neq p(i+1)$ and all other generators are even.
\end{Theorem}
The correspondence is given by
\begin{gather*}
h_i\mapsto \begin{cases}
-E_{m+n,m+n}-E_{1,1}+c&\text{ if }i=0,\\
(-1)^{p(i)}E_{i,i}-(-1)^{p(i+1)}E_{i+1,i+1}&\text{ if }i\neq 0,
\end{cases}\\
x^+_i\mapsto \begin{cases}
E_{m+n,1}t&\text{ if }i=0,\\
E_{i,i+1}&\text{ if }i\neq 0,
\end{cases}\qquad x^-_i\mapsto \begin{cases}
-E_{1,m+n}t^{-1}&\text{ if }i=0,\\
(-1)^{p(i)}E_{i+1,i}&\text{ if }i\neq 0.
\end{cases}
\end{gather*}
Let us set an anti-homomorphism of $U(\mathfrak{gl}({m|n}))$.
For associative superalgebras $X$ and $X'$ and their parities $p$ and $p'$, we say that a linear map $\pi\colon X\to X'$ is an anti-homomorphism if 
\begin{equation}
\pi(xy)=(-1)^{p(x)p(y)}\pi(y)\pi(x),\ p(x)=p'(\pi(x)),\label{anti}
\end{equation}
where $x,y$ are homogeneous elements of $X$. By the relation \eqref{anti}, we find that
\begin{equation*}
\pi([x,y])=-[\pi(x),\pi(y)].
\end{equation*}
Let us set an anti-isomorphism
\begin{equation*}
\widetilde{\omega}\colon U(\widehat{\mathfrak{gl}}(m|n))\to U(\widehat{\mathfrak{gl}}(m|n))
\end{equation*}
given by
\begin{equation*}
\widetilde{\omega}(E_{i,j}t^s)=(-1)^{\delta(i>j)(p(i)+p(j))}E_{j,i}t^{-s}.
\end{equation*}
By the definition of $\widetilde{\omega}$, we have
\begin{equation*}
\widetilde{\omega}(h_i)=h_i,\ \widetilde{\omega}(x^+_{i})=(-1)^{p(i)}x^-_{i},\ \widetilde{\omega}(x^-_{i})=(-1)^{p(i+1)}(x^+_{i}).
\end{equation*}
\section{The affine super Yangian of type $A$}
The author defines the affine super Yangian in Definition 3.1 of \cite{U2}. In this article, we use the presentation of the affine super Yangian given in Proposition 2.23 of \cite{U4}.
\begin{Definition}\label{Prop32}
Let $\hbar,\ve\in\mathbb{C}$. The affine super Yangian $Y_{\hbar,\ve}(\widehat{\mathfrak{sl}}(m|n))$ is the associative superalgebra over $\mathbb{C}$ generated by
\begin{equation*}
\{X_{i,r}^\pm, H_{i,r}\mid i\in I=\mathbb{Z}/(m+n)\mathbb{Z},r=0,1\}
\end{equation*}
subject to the following defining relations:
\begin{gather}
[H_{i,r}, H_{j,s}] = 0,\label{Eq2.1}\\
[X_{i,0}^{+}, X_{j,0}^{-}] = \delta_{i,j} H_{i, 0},\label{Eq2.2}\\
[X_{i,1}^{+}, X_{j,0}^{-}] = \delta_{i,j} H_{i, 1} = [X_{i,0}^{+}, X_{j,1}^{-}],\label{Eq2.3}\\
[H_{i,0}, X_{j,r}^{\pm}] = \pm a_{i,j} X_{j,r}^{\pm},\label{Eq2.4}\\
[\tilde{H}_{i,1}, X_{j,0}^{\pm}] = \pm a_{i,j} X_{j,1}^{\pm}\text{ if }(i,j)\neq(0,m+n-1),(m+n-1,0),\label{Eq2.5}\\
[\widetilde{H}_{0,1}, X_{m+n-1,0}^{\pm}] = \pm\left(X_{m+n-1,1}^{\pm}+(\ve+\dfrac{\hbar}{2}(m-n)\hbar) X_{m+n-1, 0}^{\pm}\right),\label{Eq2.6}\\
[\widetilde{H}_{m+n-1,1}, X_{0,0}^{\pm}] = \pm\left(X_{0,1}^{\pm}-(\ve+\dfrac{\hbar}{2}(m-n)\hbar)  X_{0, 0}^{\pm}\right),\label{Eq2.7}\\
[X_{i, 1}^{\pm}, X_{j, 0}^{\pm}] - [X_{i, 0}^{\pm}, X_{j, 1}^{\pm}] = \pm a_{i,j}\dfrac{\hbar}{2} \{X_{i, 0}^{\pm}, X_{j, 0}^{\pm}\}\text{ if }(i,j)\neq(0,m+n-1),(m+n-1,0),\label{Eq2.8}\\
[X_{0, 1}^{\pm}, X_{m+n-1, 0}^{\pm}] - [X_{0, 0}^{\pm}, X_{m+n-1, 1}^{\pm}]=\pm\dfrac{\hbar}{2} \{X_{0, 0}^{\pm}, X_{m+n-1, 0}^{\pm}\}+(\ve+\dfrac{\hbar}{2}(m-n)\hbar) [X_{0, 0}^{\pm}, X_{m+n-1, 0}^{\pm}],\label{Eq2.9}\\
(\ad X_{i,0}^{\pm})^{1+|a_{i,j}|} (X_{j,0}^{\pm})= 0 \text{ if }i \neq j, \label{Eq2.10}\\
[X^\pm_{i,0},X^\pm_{i,0}]=0\text{ if }p(i)\neq p(i+1),\label{Eq2.11}\\
[[X^\pm_{i-1,0},X^\pm_{i,0}],[X^\pm_{i,0},X^\pm_{i+1,0}]]=0\text{ if }p(i)\neq p(i+1),\label{Eq2.12}
\end{gather}
where the generators $X^\pm_{i, r}$ are odd if $p(i)\neq p(i+1)$, all other generators are even and we set $\widetilde{H}_{i,1} = H_{i,1}-\dfrac{\hbar}2 H_{i,0}^2$ and $\{X,Y\}=XY+YX$. 
\end{Definition}
\begin{Remark}
In \cite{U2}, we define the affine super Yangian in the case that $m\neq n$. In Definition~\ref{Prop32}, we also define $Y_{\hbar,\ve}(\widehat{\mathfrak{sl}}(m|n))$ in the case that $m=n$ by the same relations as the case that $m\neq n$.
\end{Remark}
 Let us set an anti-isomorphism
\begin{equation*}
\omega\colon Y_{\hbar,\ve}(\widehat{\mathfrak{sl}}(m|n))\to Y_{\hbar,\ve}(\widehat{\mathfrak{sl}}(m|n))
\end{equation*}
given by
\begin{equation*}
\omega(H_{i,r})=H_{i,r},\ \omega(X^+_{i,r})=(-1)^{p(i)}X^-_{i,r},\ \omega(X^-_{i,r})=(-1)^{p(i+1)}(X^+_{i,r}).
\end{equation*}
By Definition~\ref{Prop32} and Theorem~\ref{presentation}, we have a homomorphism from $U(\widehat{\mathfrak{sl}}(m|n))$ to $Y_{\hbar,\ve}(\widehat{\mathfrak{sl}}(m|n))$ by $h_i\mapsto H_{i,0},x^\pm_i\mapsto X^\pm_{i,0}$. In order to simplify the notation, we denote the image of $x\in U(\widehat{\mathfrak{sl}}(m|n))$ via this homomorphism by $x$. 

By using the defining relations of the affine super Yangian, we find the following relations (see Theorem~3.16 in \cite{U2} and the proof of Proposition 2.23 in \cite{U4}):
\begin{gather}
[X^\pm_{i,r},X^\pm_{j,s}]=0\text{ if }|i-j|>1.\label{gather1}
\end{gather}
We take one completion of $Y_{\hbar,\ve}(\widehat{\mathfrak{sl}}(m|n))$. We set the degree of $Y_{\hbar,\ve}(\widehat{\mathfrak{sl}}(m|n))$ by
\begin{equation*}
\text{deg}(H_{i,r})=0,\ \text{deg}(X^\pm_{i,r})=\begin{cases}
\pm1&\text{ if }i=0,\\
0&\text{ if }i\neq0.
\end{cases}
\end{equation*}
We denote the standard degreewise completion of $Y_{\hbar,\ve}(\widehat{\mathfrak{sl}}(m|n))$ by $\widetilde{Y}_{\hbar,\ve}(\widehat{\mathfrak{sl}}(m|n))$. 
Let us set $A_i\in \widetilde{Y}_{\hbar,\ve}(\widehat{\mathfrak{sl}}(m|n))$ as
\begin{align*}
A_i&=\dfrac{\hbar}{2}\sum_{\substack{s\geq0\\u>v}}\limits (-1)^{p(v)}E_{u,v}t^{-s}[(-1)^{p(i)}E_{i,i},E_{v,u}t^s]+\dfrac{\hbar}{2}\sum_{\substack{s\geq0\\u<v}}\limits (-1)^{p(v)}E_{u,v}t^{-s-1}[(-1)^{p(i)}E_{i,i},E_{v,u}t^{s+1}]\\
&=\dfrac{\hbar}{2}\sum_{\substack{s\geq0\\u>i}}\limits E_{u,i}t^{-s}E_{i,u}t^s-\dfrac{\hbar}{2}(-1)^{p(i)}\sum_{\substack{s\geq0\\i>v}}\limits (-1)^{p(v)}E_{i,v}t^{-s}E_{v,i}t^s\\
&\quad+\dfrac{\hbar}{2}\sum_{\substack{s\geq0\\u<i}}\limits E_{u,i}t^{-s-1}E_{i,u}t^{s+1}-\dfrac{\hbar}{2}(-1)^{p(i)}\sum_{\substack{s\geq0\\i<v}}\limits(-1)^{p(v)} E_{i,v}t^{-s-1}E_{v,i}t^{s+1}.
\end{align*}
Similarly to Section~3 in \cite{U2}, we define the elements of $\widetilde{Y}_{\hbar,\ve}(\widehat{\mathfrak{sl}}(m|n))$
\begin{align*}
J(h_i)&=\widetilde{H}_{i,1}+A_i-A_{i+1},\ 
J(x^\pm_i)=\begin{cases}
\mp(-1)^{p(i)}[J(h_{i-1}),x^\pm_i]&\text{ if }i\neq 0,\\
\mp[J(h_1),x^\pm_0]&\text{ if }i=0.
\end{cases}
\end{align*}

Let $\alpha$ be a positive real root of $\widehat{\mathfrak{sl}}(m|n)$. We take $x^\pm_\alpha$ be a non-zero element of the root space with $\alpha$.
\begin{Lemma}[Proposition 4.26 in \cite{U2}]\label{J}
There exists a complex number $c_{\alpha,i}$ satisfying that
\begin{equation*}
(\alpha_j,\alpha)[J(h_i),x^\pm_\alpha]-(\alpha_i,\alpha)[J(h_j),x^\pm_\alpha]=\pm c_{\alpha,i}x_\alpha^\pm,
\end{equation*}
where $(\ ,\ )$ is an inner product on $\bigoplus_{i\in I}\mathbb{C}\alpha_i$ defined by $(\alpha_i,\alpha_j)=a_{i,j}$.
\end{Lemma}

\section{The coproduct for the affine super Yangian}
The author \cite{U2} gave a coproduct for the affine super Yangian. By the same degree as $\widetilde{Y}_{\hbar,\ve}(\widehat{\mathfrak{sl}}(m|n))$, let us set $Y_{\hbar,\ve}(\widehat{\mathfrak{sl}}(m|n))\widehat{\otimes} Y_{\hbar,\ve}(\widehat{\mathfrak{sl}}(m|n))$ as the standard degreewise completion of $\otimes^2Y_{\hbar,\ve}(\widehat{\mathfrak{sl}}(m|n))$.
\begin{Theorem}[Theorem 4.3 in \cite{U2}]
There exists an algebra homomorphism
\begin{equation*}
\Delta\colon Y_{\hbar,\ve}(\widehat{\mathfrak{sl}}(m|n))\to Y_{\hbar,\ve}(\widehat{\mathfrak{sl}}(m|n))\widehat{\otimes} Y_{\hbar,\ve}(\widehat{\mathfrak{sl}}(m|n))
\end{equation*}
determined by
\begin{gather*}
\Delta(X^\pm_{j,0})=X^\pm_{j,0}\otimes1+1\otimes X^\pm_{j,0}\text{ for }0\leq j\leq m+n-1,\\
\Delta(X^+_{i,1})=X^+_{i,1}\otimes1+1\otimes X^+_{i,1}+B_i\text{ for }1\leq i\leq m+n-1,
\end{gather*}
where  and we set $B_i$ as
\begin{align*}
B_i&=\hbar\displaystyle\sum_{s \geq 0}  \limits\displaystyle\sum_{u=1}^{i}\limits ((-1)^{p(u)}E_{i,u}t^{-s}\otimes E_{u,i+1}t^s-(-1)^{p(u)+p(E_{i,u})p(E_{i+1,u})}E_{u,i+1}t^{-s-1}\otimes E_{i,u}t^{s+1})\\
&\quad+\hbar\displaystyle\sum_{s \geq 0}  \limits\displaystyle\sum_{u=i+1}^{m+n}\limits ((-1)^{p(u)}E_{i,u}t^{-s-1}\otimes E_{u,i+1}t^{s+1}-(-1)^{p(u)+p(E_{i,u})p(E_{i+1,u})}E_{u,i+1}t^{-s}\otimes E_{i,u}t^{s}).
\end{align*}
\end{Theorem}
The homomorphism $\Delta$ can be said as the coproduct for the affine super Yangian since $\Delta$ satisfies the coassociativity.
\begin{Remark}\label{remark}
In \cite{U2}, we gave $\Delta$ for $Y_{\hbar,\ve}(\widehat{\mathfrak{sl}}(m|n))$ in the case that $m\neq n$. However, by the same proof as Theorem 4.3 in \cite{U2}, the coproduct $\Delta$ can be proven in the case that $m=n$.
\end{Remark}
\section{The evaluation map for the affine super Yangian}
The evaluation map for the affine super Yangian is a a non-trivial homomorphism from the affine super Yangian $Y_{\hbar,\ve}(\widehat{\mathfrak{sl}}(m|n))$ to the completion of the universal enveloping algebra of $\widehat{\mathfrak{gl}}(m|n)$. 
We consider a completion of $U(\widehat{\mathfrak{gl}}(m|n))/U(\widehat{\mathfrak{gl}}(m|n))(z-1)$ following \cite{MNT}. 
We take the grading of $U(\widehat{\mathfrak{gl}}(m|n))/U(\widehat{\mathfrak{gl}}(m|n))(z-1)$ as $\text{deg}(Xt^s)=s$ and $\text{deg}(c)=0$. We denote the degreewise completion of $U(\widehat{\mathfrak{gl}}(m|n))/U(\widehat{\mathfrak{gl}}(m|n))(z-1)$ by $\mathcal{U}(\widehat{\mathfrak{gl}}(m|n))$.
\begin{Theorem}[Theorem 5.1 in \cite{U2} and Theorem 3.29 in \cite{U3}]\label{thm:main}
\begin{enumerate}
\item Let us set $\hat{i}=\sum_{u=1}^i\limits(-1)^{p(u)}$ for $1\leq i\leq m+n-1$. Suppose that $\hbar\neq0$ and $c=\dfrac{\ve}{\hbar}$.
For a complex number $a$, there exists an algebra homomorphism 
\begin{equation*}
\ev_{\hbar,\ve}^{m|n,a} \colon Y_{\hbar,\ve}(\widehat{\mathfrak{sl}}(m|n)) \to \mathcal{U}(\widehat{\mathfrak{gl}}(m|n))
\end{equation*}
uniquely determined by 
\begin{gather*}
\ev_{\hbar,\ve}^{m|n,a}(X_{i,0}^{+}) = \begin{cases}
E_{m+n,1}t&\text{ if }i=0,\\
E_{i,i+1}&\text{ if }1\leq i\leq m+n-1,
\end{cases} \\
\ev_{\hbar,\ve}^{m|n,a}(X_{i,0}^{-}) = \begin{cases}
-E_{1,m+n}t^{-1}&\text{ if }i=0,\\
(-1)^{p(i)}E_{i+1,i}&\text{ if }1\leq i\leq m+n-1,
\end{cases}
\end{gather*}
and
\begin{align*}
\ev_{\hbar,\ve}^{m|n,a}(X^+_{i,1}) &=(a-\dfrac{\hat{i}}{2}\hbar) E_{i,i+1}+ (-1)^{p(i)}\hbar \displaystyle\sum_{s \geq 0}  \limits\displaystyle\sum_{k=1}^{i}\limits (-1)^{p(k)}E_{i,k}t^{-s}E_{k,i+1}t^s\\
&\quad+(-1)^{p(i)}\hbar \displaystyle\sum_{s \geq 0} \limits\displaystyle\sum_{k=i+1}^{m+n}\limits (-1)^{p(k)}E_{i,k}t^{-s-1}E_{k,i+1}t^{s+1}\text{ for }i\neq0.
\end{align*}
\item In the case that $\ve\neq 0$, the image of the evaluation map is dense in $\mathcal{U}(\widehat{\mathfrak{gl}}(m|n))$.
\end{enumerate}
\end{Theorem}
\begin{Remark}\label{remark0}
In \cite{U2}, we gave the evaluation map for $Y_{\hbar,\ve}(\widehat{\mathfrak{sl}}(m|n))$ in the case that $m\neq n$. However, by the same proof as Theorem 4.3 in \cite{U2}, the evaluation map can be proven in the case that $m=n$.
\end{Remark}
\section{The first edge contraction of the affine super Yangian}
We construct four edge contractions for the affine super Yangian in this article.
We do not identify $I$ with $\mathbb{Z}/(m+n)\mathbb{Z}$ in this section. In this section, the parity $p$ is the parity of $\widehat{\mathfrak{sl}}(m+1|n)$ not $\widehat{\mathfrak{sl}}(m|n)$.
\begin{Theorem}
There exists a homomorphism 
\begin{equation*}
\Psi_1\colon Y_{\hbar,\ve}(\widehat{\mathfrak{sl}}(m|n))\to  \widetilde{Y}_{\hbar,\ve}(\widehat{\mathfrak{sl}}(m+1|n))
\end{equation*}
defined by
\begin{gather*}
\Psi_{1}(H_{i,0})=\begin{cases}
H_{m,0}+H_{m+1,0}&\text{ if }i=m,\\
H_{i,0}&\text{ if }i\neq m,
\end{cases}\\
\Psi_{1}(X^+_{i,0})=\begin{cases}
[X^+_{m,0},X^+_{m+1,0}]&\text{ if }i=m,\\
X^+_{i,0}&\text{ if }i\neq m,
\end{cases}\ 
\Psi_{1}(X^-_{i,0})=\begin{cases}
[X^-_{m+1,0},X^-_{m,0}]&\text{ if }i=m,\\
X^-_{i,0}&\text{ if }i\neq m,
\end{cases}
\end{gather*}
and
\begin{align*}
\Psi_{1}(H_{i,1})&=\begin{cases} 
H_{i,1}-\hbar\displaystyle\sum_{s \geq 0} \limits E_{i,m+1}t^{-s-1} E_{m+1,i}t^{s+1}\\
\qquad+\hbar\displaystyle\sum_{s \geq 0}\limits E_{i+1,m+1}t^{-s-1} E_{m+1,i+1}t^{s+1}\qquad\qquad\qquad\qquad\ \text{ if }1\leq i\leq m-1,\\
H_{m,1}+H_{m+1,1}+\hbar H_{m,0}H_{m+1,0}+\dfrac{\hbar}{2}H_{m+1,0}\\
\qquad-\hbar\displaystyle\sum_{s \geq 0} \limits E_{m,m+1}t^{-s-1} E_{m+1,m}t^{s+1}-\hbar\displaystyle\sum_{s \geq 0}\limits E_{-1,m+1}t^{-s} E_{m+1,-1}t^{s}\text{ if }i=m,\\
H_{i,1}+\dfrac{\hbar}{2}H_{i,0}+\hbar\displaystyle\sum_{s \geq 0} \limits E_{i,m+1}t^{-s} E_{m+1,i}t^{s}\\
\qquad-\hbar\displaystyle\sum_{s \geq 0}\limits E_{i-1,m}t^{-s} E_{m+1,i-1}t^{s}\qquad\qquad\qquad\qquad\quad\text{ if }-n+1\leq i\leq -1,\\
H_{-n,1}+\hbar\displaystyle\sum_{s \geq 0} \limits E_{-n,m+1}t^{-s} E_{m+1,-n}t^{s}\\
\qquad+\hbar\displaystyle\sum_{s \geq 0}\limits E_{1,m+1}t^{-s-1} E_{m+1,1}t^{s+1}\qquad\qquad\qquad\qquad\qquad\qquad\quad\text{ if }i=-n,
\end{cases}\\
\Psi_{1}(X^+_{i,1})&=\begin{cases}
 X^+_{i,1}-\hbar\displaystyle\sum_{s \geq 0}\limits E_{i,m+1}t^{-s-1} E_{m+1,i+1}t^{s+1}&\text{ if }1\leq i\leq m-1,\\
 [X^+_{m,1},X^+_{m+1,0}]-\hbar\displaystyle\sum_{s \geq 0}\limits E_{m,m+1}t^{-s-1} E_{m+1,-1}t^{s+1}&\text{ if }i=m,\\
 X^+_{i,1}+\dfrac{\hbar}{2}X^+_{i,0}-\hbar\displaystyle\sum_{s \geq 0}\limits E_{i,m+1}t^{-s} E_{m+1,i-1}t^{s}&\text{ if }-n+1\leq i\leq -1,\\
 X^+_{-n,1}-\hbar\displaystyle\sum_{s \geq 0}\limits E_{-n,m+1}t^{-s} E_{m+1,1}t^{s+1}&\text{ if }i=-n,
\end{cases}
\end{align*}
We define
\begin{align*}
\Psi_1(X^-_{i,1})&=\omega\circ\Psi_1(X^+_{i,1}).
\end{align*}
In particular, we have
\begin{align*}
\Psi_{1}(\widetilde{H}_{i,1})&=\begin{cases} 
\widetilde{H}_{i,1}-\hbar\displaystyle\sum_{s \geq 0} \limits E_{i,m+1}t^{-s-1} E_{m+1,i}t^{s+1}\\
\qquad+\hbar\displaystyle\sum_{s \geq 0}\limits E_{i+1,m+1}t^{-s-1} E_{m+1,i+1}t^{s+1}\qquad\qquad\qquad\qquad\ \text{ if }1\leq i\leq m-1,\\
\widetilde{H}_{m,1}+\widetilde{H}_{m+1,1}+\hbar H_{m,0}H_{m+1,0}+\dfrac{\hbar}{2}H_{m+1,0}\\
\qquad-\hbar\displaystyle\sum_{s \geq 0} \limits E_{m,m+1}t^{-s-1} E_{m+1,m}t^{s+1}-\hbar\displaystyle\sum_{s \geq 0}\limits E_{-1,m+1}t^{-s} E_{m+1,-1}t^{s}\text{ if }i=m,\\
\widetilde{H}_{i,1}+\dfrac{\hbar}{2}H_{i,0}+\hbar\displaystyle\sum_{s \geq 0} \limits E_{i,m+1}t^{-s} E_{m+1,i}t^{s}\\
\qquad-\hbar\displaystyle\sum_{s \geq 0}\limits E_{i-1,m}t^{-s} E_{m+1,i-1}t^{s}\qquad\qquad\qquad\qquad\quad\text{ if }-n+1\leq i\leq -1,\\
\widetilde{H}_{-n,1}+\hbar\displaystyle\sum_{s \geq 0} \limits E_{-n,m+1}t^{-s} E_{m+1,-n}t^{s}\\
\qquad+\hbar\displaystyle\sum_{s \geq 0}\limits E_{1,m+1}t^{-s-1} E_{m+1,1}t^{s+1}\qquad\qquad\qquad\qquad\qquad\qquad\quad\text{ if }i=-n.
\end{cases}
\end{align*}
\end{Theorem}
It is enough to show the compatibility with \eqref{Eq2.1}-\eqref{Eq2.12}. It is trivial that $\Psi_1$ is compatible with \eqref{Eq2.2}, \eqref{Eq2.4} and \eqref{Eq2.10}-\eqref{Eq2.12}. Thus, it is enough to prove the compatibility with \eqref{Eq2.1}, \eqref{Eq2.3} and \eqref{Eq2.5}-\eqref{Eq2.9}. 

By the definition of $\Psi_1$ and $\omega$, we have the relation 
\begin{gather*}
\omega\circ\Psi_1(H_{i,r})=\Psi_1(H_{i,r}),\quad\omega\circ\Psi_1(X^+_{i,r})=(-1)^{p(i)}\Psi_1(X^-_{i,r})
\end{gather*}
for $r=0,1$.
Thus, as for the compatibility with the relations \eqref{Eq2.5}-\eqref{Eq2.9}, it is enough to show the compatibility with \eqref{Eq2.5}-\eqref{Eq2.9} for $+$ by the definition of $\Psi_1(X^-_{i,1})$.
\subsection{Compatibility with \eqref{Eq2.3}}
By using the relation \eqref{Eq2.1} and \eqref{Eq2.5}-\eqref{Eq2.7}, the relation $[X^+_{i,0},X^-_{j,1}]=\delta_{i,j}H_{i,1}$ can be derived from $[X^+_{i,1},X^-_{j,0}]=\delta_{i,j}H_{i,1}$. Thus, in order to show the compatibility with \eqref{Eq2.3}, it is enough to prove the compatibility with $[X^+_{i,1},X^-_{j,0}]=\delta_{i,j}H_{i,1}$.

We only show the case that $(i,j)=(m,m)$. By the definition of $\Psi_1$, we have
\begin{align}
&\quad[\Psi_1(X^+_{m.1}),\Psi_1(X^-_{m,0})]\nonumber\\
&=[[X^+_{m,1},X^+_{m+1,0}],[X^-_{m+1,0},X^-_{m,0}]]-[\hbar\displaystyle\sum_{s \geq 0}\limits E_{m,m+1}t^{-s-1} E_{m+1,-1}t^{s+1},E_{-1,m}].\label{utyu}
\end{align}
We denote the $r$-th term of the right hand side of $(\text{equation number})$ by $(\text{equation number})_r$.
By a direct computation, we find that
\begin{align}
\eqref{utyu}_2=-\hbar\displaystyle\sum_{s \geq 0}\limits E_{m,m+1}t^{-s-1} E_{m+1,m}t^{s+1}-\hbar\displaystyle\sum_{s \geq 0}\limits E_{-1,m+1}t^{-s-1} E_{m+1,-1}t^{s+1}.\label{utyu-1}
\end{align}
By \eqref{Eq2.2}-\eqref{Eq2.5}, we obtain
\begin{align}
\eqref{utyu}_1&=[[X^+_{m,1},H_{m+1,0}],X^-_{m,0}]-[X^-_{m+1,0},[H_{m,1},X^+_{m+1,0}]]\nonumber\\
&=H_{m,1}-[[\widetilde{H}_{m,1}+\dfrac{\hbar}{2}H_{m,0}^2,X^+_{m+1,0}],X^-_{m+1,0}]\nonumber\\
&=H_{m,1}+H_{m+1,1}+\dfrac{\hbar}{2}[(H_{m,0}X^+_{m+1,0}+X^+_{m+1,0}H_{m,0}),X^-_{m+1,0}]\nonumber\\
&=H_{m,1}+H_{m+1,1}+\hbar H_{m,0}H_{m+1,0}+\dfrac{\hbar}{2}(-X^-_{m+1,0}X^+_{m+1,0}+X^+_{m+1,0}X^-_{m+1,0})\nonumber\\
&=H_{m,1}+H_{m+1,1}+\hbar H_{m,0}H_{m+1,0}-\hbar X^-_{m+1,0}X^+_{m+1,0}+\dfrac{\hbar}{2}H_{m+1,0}.\label{utyu-2}
\end{align}
By adding \eqref{utyu-1} and \eqref{utyu-2}, we have shown the relation $[\Psi_1(X^+_{m.1}),\Psi_1(X^-_{m,0})]=\Psi_1(H_{m,1})$.
\subsection{Compatibility with \eqref{Eq2.5}}
We only show the case that $(i,j)=(-1,m),(m,m)$. The other cases can be proven in a similar way.

First, we show the case that $(i,j)=(-1,m)$. By the definition of $\Psi_1$, we have
\begin{align*}
&\quad[\Psi_1(\widetilde{H}_{-1,1}),\Psi_1(X^+_{m,0})]\nonumber\\
&=[\widetilde{H}_{-1,1},[X^+_{m,0},X^+_{m+1,0}]]+\dfrac{\hbar}{2}[H_{-1,0},[X^+_{m,0},X^+_{m+1,0}]]\\
&\quad+[\hbar\displaystyle\sum_{s \geq 0} \limits E_{-1,m+1}t^{-s} E_{m+1,-1}t^{s},E_{m,-1}]-[\hbar\displaystyle\sum_{s \geq 0}\limits E_{-2,m}t^{-s} E_{m+1,-2}t^{s},E_{m,-1}]\nonumber\\
&=[X^+_{m,0},X^+_{m+1,1}]+\dfrac{\hbar}{2}E_{m,-1}-\hbar\displaystyle\sum_{s \geq 0}\limits E_{m,m+1}t^{-s} E_{m+1,-1}t^{s}\\
&=[X^+_{m,1},X^+_{m+1,0}]+\dfrac{\hbar}{2}(X^+_{m,0}X^+_{m+1,0}+X^+_{m+1,0}X^+_{m,0})+\dfrac{\hbar}{2}E_{m,-1}\\
&\quad-\hbar\displaystyle\sum_{s \geq 0}\limits E_{m,m+1}t^{-s} E_{m+1,-1}t^{s}\\
&= [X^+_{m,1},X^+_{m+1,0}]-\hbar\displaystyle\sum_{s \geq 0}\limits E_{m,m+1}t^{-s-1} E_{m+1,-1}t^{s+1},
\end{align*}
where the second equality is due to \eqref{Eq2.5} and the third equality is due to \eqref{Eq2.8}. Thus, we have shown the compatibility with \eqref{Eq2.5} in the case that $(i,j)=(-1,m)$.

Next, we show the case that $(i,j)=(m,m)$. By the definition of $\Psi_1$, we have
\begin{align*}
&\quad[\Psi_1(\widetilde{H}_{m,1}),\Psi_1(X^+_{m,0})]\nonumber\\
&=[\widetilde{H}_{m,1}+\widetilde{H}_{m+1,1},[X^+_{m,0},X^+_{m+1,0}]]+\dfrac{\hbar}{2}[H_{m+1,0},E_{m,-1}]\\
&\quad-[\hbar\displaystyle\sum_{s \geq 0} \limits E_{m,m+1}t^{-s-1} E_{m+1,m}t^{s+1},E_{m,-1}]-[\hbar\displaystyle\sum_{s \geq 0}\limits E_{-1,m+1}t^{-s} E_{m+1,-1}t^{s},E_{m,-1}]\\
&=[X^+_{m,1},X^+_{m+1,0}]-[X^+_{m,0},X^+_{m+1,1}]-\dfrac{\hbar}{2}E_{m,-1}\\
&\quad-\hbar\displaystyle\sum_{s \geq 0} \limits E_{m,m+1}t^{-s-1} E_{m+1,-1}t^{s+1}+\hbar\displaystyle\sum_{s \geq 0}\limits E_{m,m+1}t^{-s} E_{m+1,-1}t^{s}\\
&=-\dfrac{\hbar}{2}(X^+_{m,0}X^+_{m+1,0}+X^+_{m+1,0},X^+_{m,0})-\dfrac{\hbar}{2}E_{m,-1}+\hbar E_{m,m+1}E_{m+1,-1}=0,
\end{align*}
where the second equality is due to \eqref{Eq2.5} and the third equality is due to \eqref{Eq2.8}. Thus, we have shown the compatibility with \eqref{Eq2.5} in the case that $(i,j)=(m,m)$.
\subsection{Compatibility with \eqref{Eq2.6}}
By the definition of $\Psi_1$ and \eqref{Eq2.6}, we obtain
\begin{align*}
&\quad[\Psi_1(\widetilde{H}_{-n,1}),\Psi_1(X^+_{-n+1,0})]\nonumber\\
&=[\widetilde{H}_{-n,1},X^+_{-n+1,0}]+[\hbar\displaystyle\sum_{s \geq 0} \limits E_{-n,m+1}t^{-s} E_{m+1,-n}t^{s},E_{-n+1,-n}]\\
&\quad+[\hbar\displaystyle\sum_{s \geq 0}\limits E_{1,m+1}t^{-s-1} E_{m+1,1}t^{s+1},E_{-n+1,-n}]\\
&=X^+_{-n+1,1}+(\ve+\dfrac{m-n+1}{2}\hbar)X^+_{-n+1,0}-\hbar\displaystyle\sum_{s \geq 0} \limits E_{-n+1,m+1}t^{-s} E_{m+1,-n}t^{s}.
\end{align*}
\subsection{Compatibility with \eqref{Eq2.7}}
By the definition of $\Psi_1$ and \eqref{Eq2.7}, we have
\begin{align*}
&\quad[\Psi_1(\widetilde{H}_{-n+1,1}),\Psi_1(X^+_{-n,0})]\nonumber\\
&=[\widetilde{H}_{-n+1,1},X^+_{-n,0}]+[\dfrac{\hbar}{2}H_{-n+1,0}+\hbar\displaystyle\sum_{s \geq 0} \limits E_{-n+1,m+1}t^{-s} E_{m+1,-n+1}t^{s},E_{-n,1}t]\\
&\quad-[\hbar\displaystyle\sum_{s \geq 0}\limits E_{-n,m+1}t^{-s-1} E_{m+1,-n}t^{s},E_{-n,1}t]\\
&=X^+_{-n,1}-(\ve+\dfrac{m-n+1}{2}\hbar)X^+_{-n,0}+\dfrac{\hbar}{2}X^+_{-n,0}-\hbar\displaystyle\sum_{s \geq 0}\limits E_{-n,m+1}t^{-s-1} E_{m+1,1}t^{s+1}.
\end{align*}
\subsection{Compatibility with \eqref{Eq2.8}}
We only show the case that $(i,j)=(m-1,m),(-1,m)$. The other cases can be proven in a similar way.

First, we show the case that $(i,j)=(m-1,m)$. By the definition of $\Psi_1$ and \eqref{gather1}, we have
\begin{align}
&\quad[\Psi_1(X^+_{m-1,1}),\Psi_1(X^+_{m,0})]\nonumber\\
&=[X^+_{m-1,1},[X^+_{m,0},X^+_{m+1,0}]]-[\hbar\displaystyle\sum_{s \geq 0}\limits E_{m-1,m+1}t^{-s-1} E_{m+1,m}t^{s+1},E_{m,-1}]\nonumber\\
&=[[X^+_{m-1,1},X^+_{m,0}],X^+_{m+1,0}]-\hbar\displaystyle\sum_{s \geq 0}\limits E_{m-1,m+1}t^{-s-1} E_{m+1,-1}t^{s+1}\label{utyu5}
\end{align}
and
\begin{align}
&\quad[\Psi_1(X^+_{m-1,0}),\Psi_1(X^+_{m,1})]\nonumber\\
&=[X^+_{m-1,0},[X^+_{m,1},X^+_{m+1,0}]]-[E_{m-1,m},\hbar\displaystyle\sum_{s \geq 0}\limits E_{m,m+1}t^{-s-1} E_{m+1,-1}t^{s+1}]\nonumber\\
&=[[X^+_{m-1,0},X^+_{m,1}],X^+_{m+1,0}]-\hbar\displaystyle\sum_{s \geq 0}\limits E_{m-1,m+1}t^{-s-1} E_{m+1,-1}t^{s+1}.\label{utyu6}
\end{align}
By \eqref{utyu5} and \eqref{utyu6}, we obtain
\begin{align*}
&\quad[\Psi_1(X^+_{m-1,1}),\Psi_1(X^+_{m,0})]-[\Psi_1(X^+_{m-1,0}),\Psi_1(X^+_{m,1})]\\
&=-\dfrac{\hbar}{2}[\{X^+_{m-1,0},X^+_{m,0}\},X^+_{m+1,0}]=-\dfrac{\hbar}{2}\{X^+_{m-1,0},E_{m,-1}\},
\end{align*}
where the first equality is due to \eqref{Eq2.8}. Thus, we have shown the compatibility with \eqref{Eq2.8} in the case that $(i,j)=(m-1,m)$.

Next, we show the case that $(i,j)=(-1,m)$. By the definition of $\Psi_1$, we have
\begin{align}
&\quad[\Psi_1(X^+_{m,1}),\Psi_1(X^+_{-1,0})]\nonumber\\
&=[[X^+_{m,1},X^+_{m+1,0}],X^+_{-1,0}]-[\hbar\displaystyle\sum_{s \geq 0}\limits E_{m,m+1}t^{-s-1} E_{m+1,-1}t^{s+1},E_{-1,-2}]\nonumber\\
&=[[X^+_{m,1},X^+_{m+1,0}],X^+_{-1,0}]-\hbar\displaystyle\sum_{s \geq 0}\limits E_{m,m+1}t^{-s-1} E_{m+1,-2}t^{s+1}\label{utyu7}
\end{align}
and
\begin{align}
&\quad[\Psi_1(X^+_{m,0}),\Psi_1(X^+_{-1,1})]\nonumber\\
&=[[X^+_{m,0},X^+_{m+1,0}],X^+_{-1,1}]+\dfrac{\hbar}{2}[E_{m,-1},X^+_{-1,0}]-[E_{m,-1},\hbar\displaystyle\sum_{s \geq 0}\limits E_{-1,m+1}t^{-s} E_{m+1,-2}t^{s}]\nonumber\\
&=[[X^+_{m,0},X^+_{m+1,0}],X^+_{-1,1}]+\dfrac{\hbar}{2}E_{m,-2}-\hbar\displaystyle\sum_{s \geq 0}\limits E_{m,m+1}t^{-s} E_{m+1,-2}t^{s}.\label{utyu8}
\end{align}
By \eqref{Eq2.8} and \eqref{gather1}, we obtain
\begin{align}
&\quad[[X^+_{m,1},X^+_{m+1,0}],X^+_{-1,0}]-[[X^+_{m,0},X^+_{m+1,0}],X^+_{-1,1}]\nonumber\\
&=[[X^+_{m,1},X^+_{m+1,0}],X^+_{-1,0}]-[[X^+_{m,0},X^+_{m+1,0}],X^+_{-1,1}]\nonumber\\
&=[[X^+_{m,0},X^+_{m+1,1}],X^+_{-1,0}]-\dfrac{\hbar}{2}[\{X^+_{m,0},X^+_{m+1,0}\},X^+_{-1,0}]-[[X^+_{m,0},X^+_{m+1,0}],X^+_{-1,1}]\\
&=[X^+_{m,0},\dfrac{\hbar}{2}\{X^+_{m+1,0},X^+_{-1,0}\}]-\dfrac{\hbar}{2}[\{X^+_{m,0},X^+_{m+1,0}\},X^+_{-1,0}]\nonumber\\
&=\dfrac{\hbar}{2}\{E_{m,-1},X^+_{-1,0}\}-\dfrac{\hbar}{2}\{X^+_{m,0},E_{m+1,-2}\}.\label{utyu9}
\end{align}
By \eqref{utyu7}-\eqref{utyu9}, we have
\begin{align*}
&\quad[\Psi_1(X^+_{m,1}),\Psi_1(X^+_{-1,0})]-[\Psi_1(X^+_{m,0}),\Psi_1(X^+_{-1,1})]\nonumber\\
&=\dfrac{\hbar}{2}\{E_{m,-1},X^+_{-1,0}\}-\dfrac{\hbar}{2}\{X^+_{m,0},E_{m+1,-2}\}-\dfrac{\hbar}{2}E_{m,-2}+\hbar E_{m,m+1} E_{m+1,-2}\\
&=\dfrac{\hbar}{2}\{E_{m,-1},X^+_{-1,0}\}.
\end{align*}
Then, we have shown the compatibility with \eqref{Eq2.8} in the case that $(i,j)=(m,-1)$.
\subsection{Compatibility with \eqref{Eq2.9}}
By the definition of $\Psi_1$, we have
\begin{align}
&\quad[\Psi_1(X^+_{-n,1}),\Psi_1(X^+_{1-n,0})]\nonumber\\
&=[X^+_{-n,1},X^+_{1-n,0}]-[\hbar\displaystyle\sum_{s \geq 0}\limits E_{-n,m+1}t^{-s} E_{m+1,1}t^{s+1},E_{1-n,-n}]\nonumber\\
&=[X^+_{-n,1},X^+_{1-n,0}]+\hbar\displaystyle\sum_{s \geq 0}\limits E_{1-n,m+1}t^{-s}E_{m+1,1}t^{s+1}\label{utyu11}
\end{align}
and
\begin{align}
&\quad[\Psi_1(X^+_{-n,0}),\Psi_1(X^+_{1-n,1})]\nonumber\\
&=[X^+_{-n,0},X^+_{1-n,1}]+\dfrac{\hbar}{2}[X^+_{-n,0},X^+_{1-n,0}]-[E_{-n,1}t,\hbar\displaystyle\sum_{s \geq 0}\limits E_{1-n,m+1}t^{-s} E_{m+1,-n}t^{s}]\nonumber\\
&=[X^+_{-n,0},X^+_{1-n,1}]-\dfrac{\hbar}{2}E_{1-n,1}t+\hbar\displaystyle\sum_{s \geq 0}\limits E_{1-n,m+1}t^{-s} E_{m+1,1}t^{s+1}.\label{utyu12}
\end{align}
By \eqref{utyu11} and \eqref{utyu12}, we obtain
\begin{align*}
&\quad[\Psi_1(X^+_{0,1}),\Psi_1(X^+_{1-n,0})]-[\Psi_1(X^+_{0,0}),\Psi_1(X^+_{1-n,1})]\\
&=[X^+_{-n,1},X^+_{1-n,0}]-[X^+_{-n,0},X^+_{1-n,1}]+\dfrac{\hbar}{2}E_{1-n,1}t\\
&=\dfrac{\hbar}{2}\{X^+_{-n,0},X^+_{1-n,0}\}+(\ve+\dfrac{m-n+1}{2}\hbar)[X^+_{-n,0},X^+_{1-n,0}]+\dfrac{\hbar}{2}E_{1-n,1}t\\
&=\dfrac{\hbar}{2}\{X^+_{-n,0},X^+_{1-n,0}\}+(\ve+\dfrac{m-n}{2}\hbar)[X^+_{-n,0},X^+_{1-n,0}].
\end{align*}
\subsection{Compatibility with \eqref{Eq2.1}}
The compatibility with \eqref{Eq2.1} is obvious in the case that $(r,s)=(0,0),(1,0),(0,1)$. Thus, it is enough to show the case that $(r,s)=(1,1)$.

We take $1\leq b\leq m+n+1,b$ and set
\begin{align*}
a_i&=\begin{cases}
1\text{ if }1\leq i< b,\\
0\text{ if }b<i\leq m+n+1,
\end{cases}\quad
P_i=(-1)^{p(i)+p(b)}\hbar\displaystyle\sum_{s \geq 0} \limits E_{i,b}t^{-s-a_i} E_{b,i}t^{s+a_i}\text{ for }i\neq b.
\end{align*}
By the definition of $P_i$ and Lemma~\ref{J}, in order to show the compatibility of $\Psi_1$ with \eqref{Eq2.1}, it is enough to show the relation
\begin{equation}
[A_i,P_j]-[A_j,P_i]+[P_i,P_j]=0\label{concl-1}
\end{equation}
in the case that $b=m+1$. For the proof of the well-definedness of another edge contraction, we show \eqref{concl-1} for $1\leq b\leq m+n+1$.

The relation \eqref{concl-1} is trivial in the case that $i=j$. We assume that $i\neq j$.
By a direct computation, we obtain
\begin{align}
[P_i,P_j]
&=(-1)^{p(i)+p(j)+p(E_{i,b})p(E_{j,b})}\hbar^2\displaystyle\sum_{s,v \geq 0}\limits  E_{j,b}t^{-v-a_j}E_{i,j}t^{v-s+a_j-a_i}E_{b,i}t^{s+a_i}\nonumber\\
&\quad-(-1)^{p(i)+p(j)+p(E_{i,b})p(E_{j,b})}\hbar^2\displaystyle\sum_{s,v \geq 0}\limits  E_{i,b}t^{-s-a_i} E_{j,i}t^{s-v+a_i-a_j}E_{b,j}t^{v+a_j}.\label{551-0-1}
\end{align}
By the definition of $A_i$, we can divide $[A_i,P_j]$ into four pieces:
\begin{align}
[A_i,P_j]
&=[\dfrac{\hbar}{2}\sum_{\substack{s\geq0\\u>i}}\limits E_{u,i}t^{-s}E_{i,u}t^s,P_j]-[\dfrac{\hbar}{2}\sum_{\substack{s\geq0\\i>u}}\limits(-1)^{p(i)+p(u)} E_{i,u}t^{-s}E_{u,i}t^s,P_j]\nonumber\\
&\quad+[\dfrac{\hbar}{2}\sum_{\substack{s\geq0\\u<i}}\limits E_{u,i}t^{-s-1}E_{i,u}t^{s+1},P_j]-[\dfrac{\hbar}{2}\sum_{\substack{s\geq0\\i<u}}\limits (-1)^{p(i)+p(u)}E_{i,u}t^{-s-1}E_{u,i}t^{s+1},P_j].\label{9112-1}
\end{align}
We compute the right hand  side of \eqref{9112-1}. By a direct computation, we obtain
\begin{align}
&\quad[\dfrac{\hbar}{2}\sum_{\substack{s\geq0\\u>i}}\limits E_{u,i}t^{-s}E_{i,u}t^s,P_j]\nonumber\\
&=(-1)^{p(j)+p(b)}\delta(j>i)\dfrac{\hbar^2}{2}\sum_{\substack{s,v\geq0}}\limits E_{j,i}t^{-s}E_{i,b}t^{s-v-a_j}E_{b,j}t^{v+a_j}\nonumber\\
&\quad+(-1)^{p(E_{j,b})p(E_{i,j})}\delta(b>i)\dfrac{\hbar^2}{2}\sum_{\substack{s,v\geq0}}\limits E_{b,i}t^{-s}E_{j,b}t^{-v-a_j}E_{i,j}t^{s+v+a_j}\nonumber\\
&\quad-(-1)^{p(E_{j,b})p(E_{i,j})}\delta(b>i)\dfrac{\hbar^2}{2}\sum_{\substack{s,v\geq0}}\limits E_{j,i}t^{-s-v-a_j}E_{b,j}t^{v+a_j}E_{i,b}t^s\nonumber\\
&\quad-(-1)^{p(j)+p(b)}\delta(j>i)\dfrac{\hbar^2}{2}\sum_{\substack{s,v\geq0}}\limits E_{j,b}t^{-v-a_j}E_{b,i}t^{v+a_j-s}E_{i,j}t^s,\label{551-1-1}\\
&\quad-[\dfrac{\hbar}{2}\sum_{\substack{s\geq0\\i>u}}\limits (-1)^{p(i)+p(u)}E_{i,u}t^{-s}E_{u,i}t^s,P_j]\nonumber\\
&=(-1)^{p(i)+p(j)+p(E_{j,b})p(E_{i,b})}\delta(b<i)\dfrac{\hbar^2}{2}\sum_{\substack{s\geq0}}\limits E_{i,b}t^{-s}E_{j,i}t^{s-v-a_j}E_{b,j}t^{v+a_j}\nonumber\\
&\quad+(-1)^{p(i)+p(b)}\delta(i>j)\dfrac{\hbar^2}{2}\sum_{\substack{s,v\geq0}}\limits E_{i,j}t^{-s}E_{j,b}t^{-v-a_j}E_{b,i}t^{s+v+a_j}\nonumber\\
&\quad-(-1)^{p(i)+p(b)}\delta(i>j)\dfrac{\hbar^2}{2}\sum_{\substack{s,v\geq0}}\limits E_{i,b}t^{-s-v-a_j}E_{b,j}t^{v+a_j}E_{j,i}t^s\nonumber\\
&\quad-(-1)^{p(i)+p(j)+p(E_{j,b})p(E_{i,b})}\delta(b>i)\dfrac{\hbar^2}{2}\sum_{\substack{s\geq0}}\limits E_{j,b}t^{-v-a_j}E_{i,j}t^{v+a_j-s}E_{b,i}t^s,\label{551-2-1}\\
&\quad[\dfrac{\hbar}{2}\sum_{\substack{s\geq0\\u<i}}\limits E_{u,i}t^{-s-1}E_{i,u}t^{s+1},P_j]\nonumber\\
&=(-1)^{p(j)+p(b)}\delta(j<i)\dfrac{\hbar^2}{2}\sum_{\substack{s,v\geq0}}\limits E_{j,i}t^{-s-1}E_{i,b}t^{s-v+1-a_j}E_{b,j}t^{v+a_j}\nonumber\\
&\quad+(-1)^{p(E_{j,b})p(E_{i,j})}\delta(b<i)\dfrac{\hbar^2}{2}\sum_{\substack{s\geq0}}\limits E_{b,i}t^{-s-1}E_{j,b}t^{-v-a_j}E_{i,j}t^{s+v+a_j+1}\nonumber\\
&\quad-(-1)^{p(E_{j,b})p(E_{i,j})}\delta(b<i)\dfrac{\hbar^2}{2}\sum_{\substack{s\geq0}}\limits E_{j,i}t^{-s-v-a_j-1}E_{b,j}t^{v+a_j}E_{i,b}t^{s+1}\nonumber\\
&\quad-(-1)^{p(j)+p(b)}\delta(j<i)\dfrac{\hbar^2}{2}\sum_{\substack{s,v\geq0}}\limits E_{j,b}t^{-v-a_j}E_{b,i}t^{v-s+a_j-1}E_{i,j}t^{s+1},\label{551-3-1}\\
&\quad-[\dfrac{\hbar}{2}\sum_{\substack{s\geq0\\i<u}}\limits (-1)^{p(i)+p(u)}E_{i,u}t^{-s-1}E_{u,i}t^{s+1},P_j]\nonumber\\
&=(-1)^{p(i)+p(j)+p(E_{j,b})p(E_{i,b})}\delta(i<b)\dfrac{\hbar^2}{2}\sum_{\substack{s,v\geq0}}\limits E_{i,b}t^{-s-1}E_{j,i}t^{s-v-a_j+1}E_{b,j}t^{v+a_j}\nonumber\\
&\quad+(-1)^{p(i)+p(b)}\delta(i<j)\dfrac{\hbar^2}{2}\sum_{\substack{s,v\geq0}}\limits E_{i,j}t^{-s-1}E_{j,b}t^{-v-a_j}E_{b,i}t^{s+v+a_j+1}\nonumber\\
&\quad-(-1)^{p(i)+p(b)}\delta(i<j)\dfrac{\hbar^2}{2}\sum_{\substack{s\geq0}}\limits E_{i,b}t^{-s-v-a_j-1}E_{b,j}t^{v+a_j}E_{j,i}t^{s+1}\nonumber\\
&\quad-(-1)^{p(i)+p(j)+p(E_{j,b})p(E_{i,b})}\delta(i<b)\dfrac{\hbar^2}{2}\sum_{\substack{s,v\geq0}}\limits E_{j,b}t^{-v-a_j}E_{i,j}t^{v-s+a_j-1}E_{b,i}t^{s+1}.\label{551-4-1}
\end{align}
Here after, we denote $(\text{equation number})_{a,b}$ means that the value of $(\text{equation number})$ at $i=a,j=b$. Then, we can rewrite
\begin{align}
[A_i,P_j]-[A_j,P_i]&=\eqref{551-1-1}_{i,j}+\eqref{551-2-1}_{i,j}+\eqref{551-3-1}_{i,j}+\eqref{551-4-1}_{i,j}\nonumber\\
&\quad-\eqref{551-1-1}_{j,i}-\eqref{551-2-1}_{j,i}-\eqref{551-3-1}_{j,i}-\eqref{551-4-1}_{j,i}.\label{9113-1-1}
\end{align}
We prove the relation \eqref{concl-1} by dividing into three cases, that is, $i<j<b$, $i<b<j$ and $b<i<j$. We only show the case that $i<j<b$. The other cases can be proven in a similar way.

In this case, we note that $a_i=a_j=1$. By a direct computation, we can rewrite the sum of terms which contain $\delta(i<j)$ as follows:
\begin{align}
\eqref{551-1-1}_{i,j,1}-\eqref{551-2-1}_{j,i,2}&=(-1)^{p(j)+p(b)}\delta(j>i)\dfrac{\hbar^2}{2}\sum_{\substack{s,v\geq0}}\limits E_{j,i}t^{-s-v-1}E_{i,b}t^{s}E_{b,j}t^{v+1},\label{bryo-1-1}\\
\eqref{551-1-1}_{i,j,4}-\eqref{551-2-1}_{j,i,3}&=-(-1)^{p(j)+p(b)}\delta(j>i)\dfrac{\hbar^2}{2}\sum_{\substack{s,v\geq0}}\limits E_{j,b}t^{-v-1}E_{b,i}t^{-s}E_{i,j}t^{s+v+1},\label{bryo-1-2}\\
-\eqref{551-3-1}_{j,i,1}+\eqref{551-4-1}_{i,j,2}&=-(-1)^{p(i)+p(b)}\delta(i<j)\dfrac{\hbar^2}{2}\sum_{\substack{s,v\geq0}}\limits E_{i,j}t^{-s-v-1}E_{j,b}t^{s}E_{b,i}t^{v+1},\label{bryo-1-3}\\
-\eqref{551-3-1}_{j,i,4}+\eqref{551-4-1}_{i,j,3}&=(-1)^{p(i)+p(b)}\delta(i<j)\dfrac{\hbar^2}{2}\sum_{\substack{s,v\geq0}}\limits E_{i,b}t^{-v-1}E_{b,j}t^{-s}E_{j,i}t^{s+v+1}.\label{bryo-1-4}
\end{align}
By a direct computation, we obtain
\begin{align}
\eqref{bryo-1-1}+\eqref{551-1-1}_{i,j,3}&=(-1)^{p(j)+p(b)}\delta(j>i)\dfrac{\hbar^2}{2}\sum_{\substack{s\geq0}}\limits (s+1)E_{j,i}t^{-s-1}E_{i,j}t^{s+1},\label{miya}\\
\eqref{bryo-1-2}+\eqref{551-1-1}_{i,j,2}&=-(-1)^{p(j)+p(b)}\delta(j>i)\dfrac{\hbar^2}{2}\sum_{\substack{s\geq0}}\limits(s+1) E_{j,i}t^{-s-1}E_{i,j}t^{s+1}\\
\eqref{bryo-1-3}-\eqref{551-1-1}_{j,i,3}&=-(-1)^{p(i)+p(b)}\delta(i<j)\dfrac{\hbar^2}{2}\sum_{\substack{s\geq0}}\limits (s+1)E_{i,j}t^{-s-1}E_{j,i}t^{s+v+1}\\
\eqref{bryo-1-4}-\eqref{551-1-1}_{j,i,2}&=(-1)^{p(i)+p(b)}\delta(i<j)\dfrac{\hbar^2}{2}\sum_{\substack{s\geq0}}\limits (s+1)E_{i,j}t^{-s-1}E_{j,i}t^{s+1}.
\end{align}
and
\begin{align}
\eqref{551-4-1}_{i,j,4}-\eqref{551-4-1}_{j,i,1}+\eqref{551-0-1}_1&=0,\\
\eqref{551-4-1}_{i,j,1}-\eqref{551-4-1}_{j,i,4}+\eqref{551-0-1}_2&=0.\label{miya1}
\end{align}
By adding \eqref{miya}-\eqref{miya1}, we obtain \eqref{concl-1}.
\section{The second edge contraction of the affine super Yangian}
In this section, we do not identify $I$ with $\mathbb{Z}/(m+n)\mathbb{Z}$. In this subsection, the parity $p$ is the parity of $\widehat{\mathfrak{sl}}(m|n+1)$ not $\widehat{\mathfrak{sl}}(m|n)$.
\begin{Theorem}
There exists a homomorphism 
\begin{equation*}
\Psi_2\colon Y_{\hbar,\ve-\hbar}(\widehat{\mathfrak{sl}}(m|n))\to  \widetilde{Y}_{\hbar,\ve}(\widehat{\mathfrak{sl}}(m|n+1))
\end{equation*}
defined by
\begin{gather*}
\Psi_{2}(H_{i,0})=\begin{cases}
H_{i,0}&\text{ if }1\leq i\leq m-1,\\
H_{m,0}+H_{-1,0}&\text{ if }i=m,\\
H_{i-1,0}&\text{ if }-n+1\leq i\leq-1,\\
H_{-n-1,0}&\text{ if }i=-n,
\end{cases}\\
\Psi_{2}(X^+_{i,0})=\begin{cases}
E_{i,i+1}&\text{ if }1\leq i\leq m-1,\\
E_{m,-2}&\text{ if }i=m,\\
E_{i-1,i-2}&\text{ if }-n+1\leq i\leq-1,\\
E_{-n-1,1}t&\text{ if }i=-n,
\end{cases}\\ 
\Psi_{2}(X^-_{i,0})=\begin{cases}
E_{i+1,i}&\text{ if }1\leq i\leq m-1,\\
E_{-2,m}&\text{ if }i=m,\\
-E_{i-2,i-1}&\text{ if }-n+1\leq i\leq-1,\\
-E_{1,-n-1}t^{-1}&\text{ if }i=-n,
\end{cases}
\end{gather*}
and
\begin{align*}
\Psi_{2}(H_{i,1})&= 
\begin{cases}
H_{i,1}+\hbar\displaystyle\sum_{s \geq 0} \limits E_{-1,i}t^{-s} E_{i,-1}t^{s}-\hbar\displaystyle\sum_{s \geq 0} \limits E_{-1,i+1}t^{-s} E_{i+1,-1}t^{s}&\text{ if }1\leq i\leq m-1,\\
H_{m,1}+H_{-1,1}+\dfrac{\hbar}{2}H_{-1,0}+\hbar H_{-1,0}H_{m,0}\\
\quad-\hbar\displaystyle\sum_{s \geq 0} \limits E_{-1,-2}t^{-s-1} E_{-2,-1}t^{s+1}+\hbar\displaystyle\sum_{s \geq 0} \limits E_{-1,m}t^{-s}E_{m,-1}t^{s}&\text{ if }i=m,\\
H_{i-1,1}+\dfrac{\hbar}{2}H_{i-1,0}+\hbar\displaystyle\sum_{s \geq 0} \limits E_{-1,i-1}t^{-s-1} E_{i-1,-1}t^{s+1}\\
\quad-\hbar\displaystyle\sum_{s \geq 0} \limits E_{-1,i-2}t^{-s-1} E_{i-2,-1}t^{s+1}&\text{ if }1-n\leq i\leq -1,\\
H_{-n-1,1}+\hbar\displaystyle\sum_{s \geq 0} \limits E_{-1,-n-1}t^{-s-1} E_{-n-1,-1}t^{s+1}\\
\quad-\hbar\displaystyle\sum_{s \geq 0} \limits E_{-1,1}t^{-s} E_{1,-1}t^{s}&\text{ if }i=-n,\\
\end{cases}\\
\Psi_{2}(X^+_{i,1})&=\begin{cases}
X^+_{i,1}+\hbar\displaystyle\sum_{s \geq 0} \limits E_{-1,i+1}t^{-s} E_{i,-1}t^{s}&\text{ if }1\leq i\leq m-1,\\
[X^+_{m,1},X^+_{-1,0}]-\hbar\displaystyle\sum_{s \geq 0} \limits E_{-1,-2}t^{-s} E_{m,-1}t^{s}&\text{ if }i=m,\\
X^+_{i-1,1}+\dfrac{\hbar}{2}X^+_{i-1,0}-\hbar\displaystyle\sum_{s \geq 0} \limits E_{-1,i-2}t^{-s-1} E_{i-1,-1}t^{s+1}&\text{ if }1-n\leq i\leq -1,\\
X^+_{-n-1,1}+\hbar\displaystyle\sum_{s \geq 0} \limits E_{-1,1}t^{-s} E_{-n-1,-1}t^{s+1}&\text{ if }i=-n,
\end{cases}
\end{align*}
We define 
\begin{align*}
\Psi_2(X^-_{i,1})&=\omega\circ\Psi_2(X^+_{i,1}).
\end{align*}
\end{Theorem}
The well-definedness of $\Psi_2$ can be proven in a similar way to $\Psi_1$. We only show the compatibility of $\Psi_2$ with \eqref{Eq2.8} and \eqref{Eq2.1}.
\subsection{Compatibility with \eqref{Eq2.8}}
We only show the case that $(i,j)=(m,-1)$. The other cases can be proven in a similar way. It is enough to show the $+$ case by the same reason as $\Psi_1$.
By the definition of $\Psi_2$, we have
\begin{align}
&\quad[\Psi_2(X^+_{m,1}),\Psi_2(X^+_{-1,0})]\nonumber\\
&=[[X^+_{m,1},X^+_{-1,0}],X^+_{-2,0}]-[\hbar\displaystyle\sum_{s \geq 0} \limits E_{-1,-2}t^{-s} E_{m,-1}t^{s},E_{-2,-3}]\nonumber\\
&=[[X^+_{m,1},X^+_{-1,0}],X^+_{-2,0}]-\hbar\displaystyle\sum_{s \geq 0} \limits E_{-1,-3}t^{-s} E_{m,-1}t^{s}\label{893}
\end{align}
and
\begin{align}
&\quad[\Psi_2(X^+_{m,0}),\Psi_2(X^+_{-1,1})]\nonumber\\
&=[[X^+_{m,0},X^+_{-1,0}],X^+_{-2,1}]+\dfrac{\hbar}{2}[E_{m,-2},E_{-2,-3}]-[E_{m,-2},\hbar\displaystyle\sum_{s \geq 0} \limits E_{-1,-3}t^{-s-1} E_{-2,-1}t^{s+1}]\nonumber\\
&=[[X^+_{m,0},X^+_{-1,0}],X^+_{-2,1}]+\dfrac{\hbar}{2}E_{m,-3}-\hbar\displaystyle\sum_{s \geq 0} \limits E_{-1,-3}t^{-s-1} E_{m,-1}t^{s+1}.\label{894}
\end{align}
By \eqref{Eq2.8} and \eqref{gather1}, we have
\begin{align}
&\quad\eqref{893}_1-\eqref{894}_1\nonumber\\
&=[X^+_{m,0},[X^+_{-1,1},X^+_{-2,0}]]+\dfrac{\hbar}{2}\{X^+_{m,0},[X^+_{-1,0}.X^+_{-2,0}]\}\nonumber\\
&\quad-[X^+_{m,0},[X^+_{-1,1},X^+_{-2,0}]]+\dfrac{\hbar}{2}\{[X^+_{-m,0},X^+_{-1,0}],X^+_{-2,0}\}\nonumber\\
&=\dfrac{\hbar}{2}\{X^+_{m,0},E_{-1,-3}\}+\dfrac{\hbar}{2}\{E_{m,-2},X^+_{-2,0}\}.\label{895}
\end{align}
By \eqref{893}-\eqref{895}, we obtain
\begin{align*}
&\quad[\Psi_2(X^+_{m,1}),\Psi_2(X^+_{-1,0})]-[\Psi_2(X^+_{m,0}),\Psi_2(X^+_{-1,1})]\\
&=\dfrac{\hbar}{2}\{E_{m,-2},X^+_{-2,0}\}-\dfrac{\hbar}{2}E_{m,-3}+\dfrac{\hbar}{2}\{X^+_{m,0},E_{-1,-3}\}-\hbar E_{-1,-3}E_{m,-1}\\
&=\dfrac{\hbar}{2}\{E_{m,-2},X^+_{-2,0}\}.
\end{align*}
Thus, we have shown the compatibility with \eqref{Eq2.8} in the case that $(i,j)=(m,-1)$.
\subsection{Compatibility with \eqref{Eq2.1}}
The compatibility with \eqref{Eq2.1} is obvious in the case that $(r,s)=(0,0),(1,0),(0,1)$. Thus, it is enough to show the case that $(r,s)=(1,1)$.

We take $1\leq b\leq m+n+1$ and set
\begin{align*}
a_i&=\begin{cases}
0\text{ if }1\leq i<b,\\
1\text{ if }b<i\leq m+n+1,
\end{cases}\quad
Q_i=\hbar\sum_{s \geq 0}\limits E_{b,i}t^{-s-a_i}E_{i,b}t^{s+a_i}\text{ for }i\neq b.
\end{align*}
By the definition of $Q_i$ and Lemma~\ref{J}, in order to show the compatibility of $\Psi_1$ with \eqref{Eq2.1}, it is enough to show the relation
\begin{equation}
[A_i,Q_j]-[A_j,Q_i]-[Q_i,Q_j]=0\label{concl-2}
\end{equation}
in the case that $b=m+1$. For the proof of the well-definedness of another edge contraction, we will show \eqref{concl-2} for $1\leq b\leq m+n+1$.

By a direct computation, we obtain
\begin{align}
[Q_i,Q_j]
&=\hbar^2\displaystyle\sum_{s,v \geq 0}\limits E_{b,i}t^{-s-a_i}E_{i,j}t^{s-v+a_i-a_j}E_{j,b}t^{v+a_j}\nonumber\\
&\quad-\hbar^2\displaystyle\sum_{s,v \geq 0}\limits E_{b,j}t^{-v-a_j}E_{i,j}t^{v-s+a_j-a_i}E_{i,b}t^{s+a_i}.\label{550-0}
\end{align}
By the definition of $A_i$, we can divide $[A_i,Q_j]$ into 4 piecies:
\begin{align}
[A_{i},Q_j]
&=[\dfrac{\hbar}{2}\sum_{\substack{s\geq0\\u>i}}\limits E_{u,i}t^{-s}E_{i,u}t^s,Q_j]-[\dfrac{\hbar}{2}\sum_{\substack{s\geq0\\i>u}}\limits (-1)^{p(E_{i,u})}E_{i,u}t^{-s}E_{u,i}t^s,Q_j]\nonumber\\
&\quad+[\dfrac{\hbar}{2}\sum_{\substack{s\geq0\\u<i}}\limits E_{u,i}t^{-s-1}E_{i,u}t^{s+1},Q_j]-[\dfrac{\hbar}{2}\sum_{\substack{s\geq0\\i<u}}\limits (-1)^{p(E_{i,u})}E_{i,u}t^{-s-1}E_{u,i}t^{s+1},Q_j].\label{9112}
\end{align}
We compute the right hand  side of \eqref{9112}. By a direct computation, we obtain
\begin{align}
\eqref{9112}_1
&=\dfrac{\hbar^2}{2}\delta(b>i)\sum_{\substack{s,v\geq0}}\limits E_{b,i}t^{-s}E_{i,j}t^{s-v-a_j}E_{j,b}t^{v+a_j}\nonumber\\
&\quad-\dfrac{\hbar^2}{2}\delta(j>i)\sum_{\substack{s,v\geq0}}\limits E_{b,i}t^{-s-v-a_j}E_{i,j}t^sE_{j,b}t^{v+a_j}\nonumber\\
&\quad+\dfrac{\hbar^2}{2}\delta(j>i)\sum_{\substack{s,v\geq0}}\limits E_{b,j}t^{-v-a_j}E_{j,i}t^{-s}E_{i,b}t^{s+v+a_j}\nonumber\\
&\quad-\dfrac{\hbar^2}{2}\delta(b>i)\sum_{\substack{s,v\geq0}}\limits E_{b,j}t^{-v-a_j}E_{j,i}t^{v-s+a_j}E_{i,b}t^s\label{550-1}\\
\eqref{9112}_2
&=\dfrac{\hbar^2}{2}\delta(i>j)\sum_{\substack{s\geq0}}\limits (-1)^{p(E_{i,j})+p(E_{i,j})p(E_{b,j})}E_{i,j}t^{-s}E_{b,i}t^{s-v-a_j}E_{j,b}t^{v+a_j}\nonumber\\
&\quad-\dfrac{\hbar^2}{2}\delta(i>b)\sum_{\substack{s\geq0}}\limits (-1)^{p(E_{i,b})+p(E_{i,b})p(E_{b,j})}E_{i,j}t^{-s-v-a_j}E_{b,i}t^sE_{j,b}t^{v+a_j}\nonumber\\
&\quad+\dfrac{\hbar^2}{2}\delta(i>b)\sum_{\substack{s\geq0}}\limits (-1)^{p(E_{i,b})+p(E_{i,b})p(E_{b,j})}E_{b,j}t^{-v-a_j}E_{i,b}t^{-s}E_{j,i}t^{s+v+a_j}\nonumber\\
&\quad-\dfrac{\hbar^2}{2}\delta(i>j)\sum_{\substack{s\geq0}}\limits (-1)^{p(E_{i,j})+p(E_{i,j})p(E_{b,j})}E_{b,j}t^{-v-a_j}E_{i,b}t^{v-s+a_j}E_{j,i}t^s\label{550-2}\\
\eqref{9112}_3
&=\dfrac{\hbar^2}{2}\delta(b<i)\sum_{\substack{s\geq0}}\limits E_{b,i}t^{-s-1}E_{i,j}t^{s-v-a_j+1}E_{j,b}t^{v+a_j}\nonumber\\
&\quad-\dfrac{\hbar^2}{2}\delta(j<i)\sum_{\substack{s\geq0}}\limits E_{b,i}t^{-s-v-a_j-1}E_{i,j}t^{s+1}E_{j,b}t^{v+a_j}\nonumber\\
&\quad+\dfrac{\hbar^2}{2}\delta(j<i)\sum_{\substack{s\geq0}}\limits E_{b,j}t^{-v-a_j}E_{j,i}t^{-s-1}E_{i,b}t^{s+v+a_j+1}\nonumber\\
&\quad-\dfrac{\hbar^2}{2}\delta(b<i)\sum_{\substack{s\geq0}}\limits E_{b,j}t^{-v-a_j}E_{j,i}t^{v-s+a_j-1}E_{i,b}t^{s+1}\label{550-3}\\
\eqref{9112}_4
&=\dfrac{\hbar^2}{2}\delta(i<j)\sum_{\substack{s\geq0}}\limits (-1)^{p(E_{i,j})+p(E_{i,j})p(E_{b,j})}E_{i,j}t^{-s-1}E_{b,i}t^{s-v+1-a_j}E_{j,b}t^{v+a_j}\nonumber\\
&\quad-\dfrac{\hbar^2}{2}\delta(i<b)\sum_{\substack{s\geq0}}\limits(-1)^{p(E_{i,b})+p(E_{i,b})p(E_{b,j})}E_{i,j}t^{-s-v-2}E_{b,i}t^{s+1}E_{j,b}t^{v+a_j}\nonumber\\
&\quad+\dfrac{\hbar^2}{2}\delta(i<b)\sum_{\substack{s\geq0}}\limits (-1)^{p(E_{i,b})+p(E_{i,b})p(E_{b,j})}E_{b,j}t^{-v-a_j}E_{i,b}t^{-s-1}E_{j,i}t^{s+v+a_j+1}\nonumber\\
&\quad-\dfrac{\hbar^2}{2}\delta(i<j)\sum_{\substack{s\geq0}}\limits(-1)^{p(E_{i,j})+p(E_{i,j})p(E_{b,j})}E_{b,j}t^{-v-a_j}E_{i,b}t^{v-s+a_j-1}E_{j,i}t^{s+1}\label{550-4}
\end{align}
Then, we find that
\begin{align*}
[A_i,Q_j]-[A_j,Q_i]&=\eqref{550-1}_{i,j}+\eqref{550-2}_{i,j}+\eqref{550-3}_{i,j}+\eqref{550-4}_{i,j}\\
&\quad-\eqref{550-1}_{j,i}-\eqref{550-2}_{j,i}-\eqref{550-3}_{j,i}-\eqref{550-4}_{j,i}.
\end{align*}
We divide the proof into three cases, that is, $i<j<b$, $i<b<j$ and $b<i<j$. We only show the case that $i<j<b$. The other cases can be proven in the same way.

By a direct computation, we can rewrite the following sum:
\begin{align}
\eqref{550-1}_{i,j,1}-\eqref{550-1}_{j,i,4}-\eqref{550-0}_1&=0,\label{miya3}\\
-\eqref{550-1}_{j,i,1}+\eqref{550-1}_{i,j,4}-\eqref{550-0}_2&=0,\\
\eqref{550-1}_{i,j,2}-\eqref{550-2}_{j,i,4}-\eqref{550-4}_{j,i,3}&=-\dfrac{\hbar^2}{2}\sum_{s\geq0}\limits (s+1)E_{b,i}t^{-s}E_{i,b}t^{s},\\
\eqref{550-1}_{i,j,3}-\eqref{550-2}_{j,i,1}-\eqref{550-4}_{j,i,2}&=\dfrac{\hbar^2}{2}\sum_{s\geq0}\limits (s+1)E_{b,i}t^{-s}E_{i,b}t^{s},\\
-\eqref{550-3}_{j,i,2}+\eqref{550-4}_{i,j,4}+\eqref{550-4}_{i,j,3}&=\dfrac{\hbar^2}{2}\sum_{s\geq0}\limits (s+1)E_{b,i}t^{-s-1}E_{i,b}t^{s+1},\\
-\eqref{550-3}_{j,i,3}+\eqref{550-4}_{i,j,1}+\eqref{550-4}_{i,j,2}&=-\dfrac{\hbar^2}{2}\sum_{s\geq0}\limits (s+1)E_{b,i}t^{-s-1}E_{i,b}t^{s+1}.\label{miya4}
\end{align}
By adding \eqref{miya3}-\eqref{miya4}, we obtain \eqref{concl-2}.
\section{The third edge contraction of the affine super Yangian}
In this section, we identify $I$ with $\mathbb{Z}/(m+n)\mathbb{Z}$. In this section, the parity $p$ is the parity of $\widehat{\mathfrak{sl}}(m+1|n)$ not $\widehat{\mathfrak{sl}}(m|n)$.
\begin{Theorem}
There exists a homomorphism 
\begin{equation*}
\Psi_3\colon Y_{\hbar,\ve+\hbar}(\widehat{\mathfrak{sl}}(m|n))\to  \widetilde{Y}_{\hbar,\ve}(\widehat{\mathfrak{sl}}(m+1|n))
\end{equation*}
defined by
\begin{gather*}
\Psi_{3}(H_{i,0})=\begin{cases}
H_{0,0}+H_{1,0}&\text{ if }i=0,\\
H_{i+1,0}&\text{ if }i\neq 0,
\end{cases}\\
\Psi_{3}(X^+_{i,0})=\begin{cases}
E_{m+n+1,2}t&\text{ if }i=0,\\
E_{i+1,i+2}&\text{ if }i\neq 0,
\end{cases}\ 
\Psi_{3}(X^-_{i,0})=\begin{cases}
-E_{2,m+n+1}t^{-1}&\text{ if }i=0,\\
(-1)^{p(i+1)}E_{i+2,i+1}&\text{ if }i\neq 0,
\end{cases}
\end{gather*}
and
\begin{align*}
\Psi_{3}(H_{i,1})&=
\begin{cases}
H_{i+1,1}+\hbar\displaystyle\sum_{s \geq 0} \limits E_{1,i+1}t^{-s-1} E_{i+1,1}t^{s+1}-\hbar\displaystyle\sum_{s \geq 0}\limits E_{1,i+2}t^{-s-1} E_{i+2,1}t^{s+1}\\
\qquad\qquad\qquad\qquad\qquad\qquad\qquad\qquad\qquad\qquad\text{ if }1\leq i\leq m+n-1,\\
H_{0,1}+H_{1,1}+\hbar H_{0,0}H_{1,0}+\dfrac{\hbar}{2}H_{0,0}\\
\quad-\hbar\displaystyle\sum_{s \geq 0} \limits E_{1,2}t^{-s-1}E_{2,1}t^{s+1}+\hbar\displaystyle\sum_{s \geq 0} \limits E_{1,m+n+1}t^{-s-1}E_{m+n+1,1}t^{s+1}\text{ if }i=0,
\end{cases}\\
\Psi_{2}(X^+_{i,1})&=
\begin{cases}
 X^+_{i+1,1}+(-1)^{p(i+1)+p(E_{i+1,1})p(E_{i+1,i+2})}\hbar\displaystyle\sum_{s \geq 0}\limits E_{1,i+2}t^{-s-1} E_{i+1,1}t^{s+1}\\
\qquad\qquad\qquad\qquad\qquad\qquad\qquad\qquad\qquad\qquad\text{ if }1\leq i\leq m+n-1,\\
[X^+_{0,0},X^+_{1,1}]+\hbar\displaystyle\sum_{s \geq 0} \limits E_{1,2}t^{-s-1}E_{m+n+1,1}t^{s+2}\text{ if }i=0,
\end{cases}
\end{align*}
We define $\Psi_3(X^-_{i,1})=\omega\circ\Psi_3(X^+_{i,1})$.
\end{Theorem}
The well-definedness of $\Psi_3$ can be proven in a similar way to $\Psi_2$. We only show the compatibility with \eqref{Eq2.9} and \eqref{Eq2.1}.
\subsection{Compatibility of \eqref{Eq2.9}}
By the definition of $\Psi_3$, we have
\begin{align}
&\quad[ \Psi_3(X_{0, 0}^+),\Psi_3(X_{m+n-1, 1}^+)]\nonumber\\
&=[[X^+_{0,0},X^+_{1,0}],X^+_{m+n,1}]-[E_{m+n+1,2}t,\hbar\displaystyle\sum_{s \geq 0}\limits E_{1,m+n+1}t^{-s-1} E_{m+n,1}t^{s+1}]\nonumber\\
&=[[X^+_{0,0},X^+_{1,0}],X^+_{m+n,1}]-\hbar\displaystyle\sum_{s \geq 0}\limits E_{1,2}t^{-s} E_{m+n,1}t^{s+1}\label{kryo-1}
\end{align}
and
\begin{align}
&\quad[\Psi_3(X_{0, 1}^+), \Psi_3(X_{m+n-1, 0}^+)]\nonumber\\
&=[[X^+_{0,0},X^+_{1,1}],X^+_{m+n,0}]+[\hbar\displaystyle\sum_{s \geq 0} \limits E_{1,2}t^{-s-1}E_{m+n+1,1}t^{s+2},E_{m+n,m+n+1}]\nonumber\\
&=[[X^+_{0,0},X^+_{1,1}],X^+_{m+n,0}]-\hbar\displaystyle\sum_{s \geq 0} \limits E_{1,2}t^{-s-1}E_{m+n,1}t^{s+2}.\label{kryo-2}
\end{align}
By \eqref{gather1}, \eqref{Eq2.8} and \eqref{Eq2.9}, we obtain
\begin{align}
&\quad\eqref{kryo-1}_1\nonumber\\
&=[[X^+_{0,0},X^+_{m+n,1}],X^+_{1,0}]\nonumber\\
&=[[X^+_{0,1},X^+_{m+n,0}],X^+_{1,0}]-\dfrac{\hbar}{2}\{[X^+_{0,0},X^+_{1,0}],X^+_{m+n,0}\}-(\ve+\dfrac{\hbar}{2}(m-n+1))[[X^+_{0,0},X^+_{m+n,0}],X^+_{1,0}]\nonumber\\
&=[[X^+_{0,1},X^+_{m+n,0}],X^+_{1,0}]-\dfrac{\hbar}{2}\{E_{m+n+1,2}t,X^+_{m+n,0}\}+(\ve+\dfrac{\hbar}{2}(m-n+1))E_{m+n,2}t\label{kryo-3}
\end{align}
and
\begin{align}
\eqref{kryo-2}_1&=[[X^+_{0,1},X^+_{1,0}],X^+_{m+n,0}]+\dfrac{\hbar}{2}\{[X^+_{0,0},X^+_{m+n,0}],X^+_{1,0}\}\nonumber\\
&=[[X^+_{0,1},X^+_{m+n,0}],X^+_{1,0}]-\dfrac{\hbar}{2}\{E_{m+n,1}t,X^+_{1,0}\}.\label{kryo-4}
\end{align}
By applying \eqref{kryo-3} and \eqref{kryo-4} to \eqref{kryo-1} and \eqref{kryo-2}, we have
\begin{align*}
&\quad[\Psi_3(X_{0, 1}^{\pm}), \Psi_3(X_{m+n-1, 0}^{\pm})]-[ \Psi_3(X_{0, 0}^+),\Psi_3(X_{m+n-1, 1}^+)]\\
&=-\dfrac{\hbar}{2}\{E_{m+n,1}t,X^+_{1,0}\}+\dfrac{\hbar}{2}\{E_{m+n+1,2}t,X^+_{m+n,0}\}-(\ve+\dfrac{\hbar}{2}(m-n+1))E_{m+n,2}t+\hbar E_{1,2}E_{m+n,1}t\\
&=\dfrac{\hbar}{2}\{E_{m+n+1,2}t,X^+_{m+n,0}\}-(\ve+\dfrac{\hbar}{2}(m-n+2))E_{m+n,2}t.
\end{align*}
Thus, we have proven the compatibility with \eqref{Eq2.9}.
\subsection{Compatibility with \eqref{Eq2.1}}
By Lemma~\ref{J}, it is enough to show the relation \eqref{concl-2} for the case that $b=1$. Thus, we have already shown in the proof of the well-definedness of $\Psi_2$.
\section{The 4-th edge contraction of the affine super Yangian}
In this section, we identify $I$ with $\mathbb{Z}/(m+n)\mathbb{Z}$. In this section, the parity $p$ is the parity of $\widehat{\mathfrak{sl}}(m|n+1)$ not $\widehat{\mathfrak{sl}}(m|n)$.
\begin{Theorem}
There exists a homomorphism 
\begin{equation*}
\Psi_4\colon Y_{\hbar,\ve}(\widehat{\mathfrak{sl}}(m|n))\to  \widetilde{Y}_{\hbar,\ve}(\widehat{\mathfrak{sl}}(m|n+1))
\end{equation*}
defined by
\begin{gather*}
\Psi_{4}(H_{i,0})=\begin{cases}
H_{0,0}+H_{m+n,0}&\text{ if }i=0,\\
H_{i,0}&\text{ if }i\neq 0,
\end{cases}\\
\Psi_{4}(X^+_{i,0})=\begin{cases}
E_{m+n,1}t&\text{ if }i=0,\\
E_{i,i+1}&\text{ if }i\neq 0,
\end{cases}\ 
\Psi_{4}(X^-_{i,0})=\begin{cases}
-E_{1,m+n}t^{-1}&\text{ if }i=0,\\
(-1)^{p(i)}E_{i+1,i}&\text{ if }i\neq 0,
\end{cases}
\end{gather*}
and
\begin{align*}
\Psi_{4}(H_{i,1})&= 
\begin{cases}
H_{i,1}+(-1)^{p(i)}\hbar\displaystyle\sum_{s \geq 0} \limits E_{i,m+n+1}t^{-s-1} E_{m+n+1,i}t^{s+1}\\
\quad-(-1)^{p(i+1)}\hbar\displaystyle\sum_{s \geq 0}\limits E_{i+1,m+n+1}t^{-s-1} E_{m+n+1,i+1}t^{s+1}\text{ if }1\leq i\leq m+n-1,\\
H_{0,1}+H_{m+n,1}+(\ve+\dfrac{\hbar}{2}(m-n))H_{m+n,0}+\hbar H_{m+n,0}H_{0,0}\\
\quad-\hbar\displaystyle\sum_{s \geq 0} \limits E_{m+n,m+n+1}t^{-s-1} E_{m+n+1,n}t^{s+1}-\hbar\displaystyle\sum_{s \geq 0} \limits E_{1,m+n+1}t^{-s-1} E_{m+n+1,1}t^{s+1}\\
\qquad\qquad\qquad\qquad\qquad\qquad\qquad\qquad\qquad\qquad\qquad\qquad\qquad\text{ if }i=0,
\end{cases}\\
\Psi_{4}(X^+_{i,1})&=
\begin{cases}
X^+_{i,1}+\hbar\displaystyle\sum_{s \geq 0}\limits E_{i,m+n+1}t^{-s-1} E_{m+n+1,i+1}t^{s+1}&\text{ if }1\leq i\leq m+n-1,\\
[X^+_{m+n,0},X^+_{0,1}]+\hbar\displaystyle\sum_{s \geq 0} \limits E_{m+n,m+n+1}t^{-s} E_{m+n+1,1}t^{s+1}&\text{ if }i=0,
\end{cases}
\end{align*}
We define $\Psi_4(X^-_{i,1})=\omega\circ\Psi_4(X^+_{i,1})$.
\end{Theorem}
The proof can be given in a similar way to $\Psi_1$. We only show the compatibility with \eqref{Eq2.9} and \eqref{Eq2.1}.
\subsection{Compatibility with \eqref{Eq2.9}}
By \eqref{gather1} and \eqref{Eq2.9}, we have
\begin{align}
&\quad[\Psi_4(X^+_{0,1}),\Psi_4(X^+_{m+n-1,0})]\nonumber\\
&=[[X^+_{m+n,0},X^+_{0,1}],X^+_{m+n-1,0}]+[\hbar\displaystyle\sum_{s \geq 0} \limits E_{m+n,m+n+1}t^{-s} E_{m+n+1,1}t^{s+1},E_{m+n-1,m+n}]\nonumber\\
&=[[X^+_{m+n,1},X^+_{0,0}],X^+_{m+n-1,0}]-\dfrac{\hbar}{2}[\{X^+_{m+n,0},X^+_{0,0}\},X^+_{m+n-1,0}]\nonumber\\
&\quad+(\ve+\dfrac{m-n-1}{2}\hbar)[[X^+_{m+n,0},X^+_{0,0}],X^+_{m+n-1,0}]-\hbar\displaystyle\sum_{s \geq 0} \limits E_{m+n-1,m+n+1}t^{-s} E_{m+n+1,1}t^{s+1}\nonumber\\
&=[[X^+_{m+n,1},X^+_{m+n-1,0}],X^+_{0,0}]+\dfrac{\hbar}{2}\{E_{m+n-1,m+n+1},X^+_{0,0}\}-(\ve+\dfrac{m-n-1}{2}\hbar)E_{m+n-1,1}t\nonumber\\
&\quad-\hbar\displaystyle\sum_{s \geq 0} \limits E_{m+n-1,m+n+1}t^{-s} E_{m+n+1,1}t^{s+1}.\label{hryo-1}
\end{align}
Similarly, by \eqref{gather1} and \eqref{Eq2.8}, we have
\begin{align}
&\quad[\Psi_4(X^+_{0,0}),\Psi_4(X^+_{m+n-1,1})]\nonumber\\
&=[[X^+_{m+n,0},X^+_{0,0}],X^+_{m+n-1,1}]+[E_{m+n,1}t,\hbar\displaystyle\sum_{s \geq 0}\limits E_{m+n-1,m+n+1}t^{-s-1} E_{m+n+1,m+n}t^{s+1}]\nonumber\\
&=[[X^+_{m+n,0},X^+_{m+n-1,1}],X^+_{0,0}]-\hbar\displaystyle\sum_{s \geq 0}\limits E_{m+n-1,m+n+1}t^{-s-1} E_{m+n+1,1}t^{s+2}.\label{hryo-2}
\end{align}
By \eqref{hryo-1} and \eqref{hryo-2}, we have
\begin{align*}
&\quad[\Psi_4(X^+_{0,1}),\Psi_4(X^+_{m+n-1,0})]-[\Psi_4(X^+_{0,0}),\Psi_4(X^+_{m+n-1,1})]\\
&=[[X^+_{m+n,1},X^+_{m+n-1,0}]-[X^+_{m+n,0},X^+_{m+n-1,1}],X^+_{0,0}]\\
&\quad+\dfrac{\hbar}{2}\{E_{m+n-1,m+n+1},X^+_{0,0}\}-(\ve+\dfrac{m-n-1}{2}\hbar)E_{m+n-1,1}t-\hbar E_{m+n-1,m+n+1}E_{m+n+1,1}t\\
&=\dfrac{\hbar}{2}\{[X^+_{m+n,0},X^+_{0,0}],X^+_{m+n-1,0}\}-(\ve+\dfrac{m-n}{2}\hbar)E_{m+n-1,1}t.
\end{align*}
Thus, we have proven the compatibility with \eqref{Eq2.9}.
\subsection{Compatibility with \eqref{Eq2.1}}
By Lemma~\ref{J}, it is enough to show the relation \eqref{concl-1} for the case that $b=m+n+1$. Thus, we have already shown in the proof of the well-definedness of $\Psi_1$.
\section{$W$-superalgebras of type $A$}

Let us set some notations of a vertex superalgebra. For a vertex superalgebra $V$, we denote the generating field associated with $v\in V$ by $v(z)=\displaystyle\sum_{s\in\mathbb{Z}}\limits v_{(s)}z^{-n-1}$. We also denote the OPE of $V$ by
\begin{equation*}
u(z)v(w)\sim\displaystyle\sum_{s\geq0}\limits \dfrac{(u_{(s)}v)(w)}{(z-w)^{s+1}}
\end{equation*}
for all $u, v\in V$. We denote the vacuum vector (resp.\ the translation operator) by $|0\rangle$ (resp.\ $\partial$).

We denote the universal affine vertex superalgebra associated with a finite dimensional Lie superalgebra $\mathfrak{g}$ and its inner product $\kappa$ by $V^\kappa(\mathfrak{g})$. By the PBW theorem, we can identify $V^\kappa(\mathfrak{g})$ with $U(t^{-1}\mathfrak{g}[t^{-1}])$. In order to simplify the notation, here after, we denote the generating field $(ut^{-1})(z)$ as $u(z)$. By the definition of $V^\kappa(\mathfrak{g})$, the generating fields $u(z)$ and $v(z)$ satisfy the OPE
\begin{gather}
u(z)v(w)\sim\dfrac{[u,v](w)}{z-w}+\dfrac{\kappa(u,v)}{(z-w)^2}\label{OPE1}
\end{gather}
for all $u,v\in\mathfrak{g}$. 

We take a positive integer and its partition:
\begin{align}
M&=\displaystyle\sum_{i=1}^lu_i,\qquad u_1\geq u_{2}\geq\cdots\geq u_l\geq0,\nonumber\\
N&=\displaystyle\sum_{i=1}^lq_i,\qquad q_1\geq q_{2}\geq\cdots\geq q_l\geq0,\label{cond:q}
\end{align}
satisfying that $M\neq N$ and $u_l+q_l\neq0$. We define a set
\begin{equation*}
I_{M|N}=\{1,2,\cdots,M,-1,-2,\cdots,-N\}.
\end{equation*}Let us set a parity of $i\in I_{M|N}$ as
\begin{equation*}
p(i)=\begin{cases}
0&\text{ if }i>0,\\
1&\text{ if }i<0.
\end{cases}
\end{equation*}
For $1\leq i\leq M$ and $-N\leq j\leq -1$, we set $1\leq \col(i),\col(j)\leq l$, $u_1-u_{\col(i)}<\row(i)\leq u_1$ and $-q_1\leq\row(j)<-q_1+q_{\col(j)}$ satisfying
\begin{gather*}
\sum_{b=1}^{\col(i)-1}u_b<i\leq\sum_{b=1}^{\col(i)}u_b,\ 
\row(i)=i-\sum_{b=1}^{\col(i)-1}u_b+u_1-u_{\col(i)},\\
\sum_{b=1}^{\col(j)-1}q_b<-j\leq\sum_{b=1}^{\col(j)}q_b,\ \row(j)=j+\sum_{b=1}^{\col(j)-1}q_b-q_1+q_{\col(j)}.
\end{gather*}
Let us set a Lie superalgebra $\mathfrak{gl}(M|N)=\bigoplus_{i,j\in I_{M|N}}\mathbb{C}e_{i,j}$ whose commutator relations are determined by
\begin{equation*}
[e_{i,j},e_{x,y}]=\delta_{j,x}e_{i,y}-(-1)^{p(e_{i,j})p(e_{x,y})}\delta_{i,y}e_{x,j},
\end{equation*}
where $p(e_{i,j})=p(i)+p(j)$.
We take a nilpotent element $f\in\mathfrak{gl}(M|N)$ as 
\begin{equation*}
f=\sum_{i\in I_{M|N}}\limits e_{\hat{i},i},
\end{equation*}
where the integer $\hat{i}\in I_{M|N}$ are determined by
\begin{gather*}
\col(\hat{i})=\col(i)+1,\ \row(\hat{i})=\row(i).
\end{gather*}
Similarly to $\hat{i}$, we set $\tilde{i}\in I_{M|N}$ as
\begin{gather*}
\col(\tilde{i})=\col(i)-1,\ \row(\tilde{i})=\row(i).
\end{gather*}
We also fix an inner product of the Lie superalgebra $\mathfrak{gl}(M|N)$ determined by
\begin{equation*}
(e_{i,j}|e_{x,y})=k\delta_{i,y}\delta_{x,j}(-1)^{p(i)}+\delta_{i,j}\delta_{x,y}(-1)^{p(i)+p(x)}.
\end{equation*}
We set two Lie superalgebra 
\begin{equation*}
\mathfrak{b}=\bigoplus_{\substack{i,j\in I_{M|N},\\\col(i)\geq\col(j)}}\mathbb{C}e_{i,j},\ \mathfrak{a}=\mathfrak{b}\oplus\bigoplus_{\substack{i,j\in I_{M|N},\\\col(i)>\col(j)}}\mathbb{C}\psi_{i,j},
\end{equation*}
whose commutator relations are defined by
\begin{gather*}
[e_{i,j},\psi_{x,y}]=\delta_{j,x}\psi_{i,y}-\delta_{i,y}(-1)^{p(e_{i,j})(p(e_{x,y})+1)}\psi_{x,j},\\
[\psi_{i,j},\psi_{x,y}]=\delta_{j,x}\psi_{i,y}-\delta_{i,y}(-1)^{(p(e_{i,j})+1)(p(e_{x,y})+1)}\psi_{x,j},
\end{gather*}
where the parity of $e_{i,j}$ is $p(i)+p(j)$ and the parity of $\psi_{i,j}$ is $p(i)+p(j)+1$. 
We also set an inner product on $\mathfrak{b}$ and $\mathfrak{a}$ by
\begin{equation*}
\kappa(e_{i,j},e_{p,q})=(e_{i,j}|e_{p,q}),\ \kappa(e_{i,j},\psi_{p,q})=\kappa(\psi_{i,j},\psi_{p,q})=0.
\end{equation*}
We denote the universal affine vertex superalgebras associated with $\mathfrak{b}$ and $\mathfrak{a}$ by $V^\kappa(\mathfrak{b})$ and $V^\kappa(\mathfrak{b})$. We also sometimes denote the elements $e_{i,j}t^{-s}\in V^\kappa(\mathfrak{b})\subset V^\kappa(\mathfrak{a})$ and $\psi_{i,j}t^{-s}\in V^\kappa(\mathfrak{a})$ by $e_{i,j}[-s]$ and $\psi_{i,j}[-s]$ respectively.
Let us define an odd differential $d_0 \colon V^{\kappa}(\mathfrak{b})\to V^{\kappa}(\mathfrak{a})$ determined by
\begin{gather}
d_01=0,\\
[d_0,\partial]=0,\label{ee5800}
\end{gather}
\begin{align}
d_0(e_{i,j}[-1]])
&=\sum_{\substack{\col(i)>\col(r)\geq\col(j)}}\limits (-1)^{p(e_{i,j})+p(e_{i,r})p(e_{r,j})}e_{r,j}[-1]\psi_{i,r}[-1]\nonumber\\
&\quad-\sum_{\substack{\col(j)<\col(r)\leq\col(i)}}\limits (-1)^{p(e_{i,r})p(e_{r,j})}\psi_{r,j}[-1]e_{i,r}[-1]\nonumber\\
&\quad+\delta(\col(i)>\col(j))(-1)^{p(i)}\alpha_{\col(i)}\psi_{i,j}[-2]+(-1)^{p(i)}\psi_{\hat{i},j}[-1]-(-1)^{p(i)}\psi_{i,\tilde{j}}[-1].\label{ee1}
\end{align}
By using Theorem 2.4 in \cite{KRW}, we can define the $W$-algebra $\mathcal{W}^k(\mathfrak{gl}(M|N),f)$ as follows.
\begin{Definition}\label{T125}
The $W$-algebra $\mathcal{W}^k(\mathfrak{gl}(M|N),f)$ is the vertex subalgebra of $V^\kappa(\mathfrak{b})$ defined by
\begin{equation*}
\mathcal{W}^k(\mathfrak{gl}(M|N),f)=\{y\in V^\kappa(\mathfrak{b})\mid d_0(y)=0\}.
\end{equation*}
\end{Definition}
We define the set
\begin{equation*}
I_s=\{1,\cdots,u_s,-1,\cdots,-q_s\}.
\end{equation*}
We construct two kinds of elements $W^{(1)}_{a,b}$ and $W^{(2)}_{a,b}$ for $a,b\in I_s\setminus I_{s+1}$. Let us set
\begin{gather*}
\alpha_s=k+M-N-u_s+q_s,\ \gamma_a=\sum_{s=a+1}^{l}\limits \alpha_{s}.
\end{gather*}

\begin{Theorem}\label{Generators1}
Let us set
\begin{align*}
W^{(1)}_{a,b}&=\sum_{\substack{\row(i)=a,\row(j)=b,\\\col(i)=\col(j)}}e_{i,j}[-1],\\
W^{(2)}_{a,b}&=\sum_{\substack{\col(i)=\col(j)+1\\\row(i)=a,\row(j)=b}}e_{i,j}[-1]-\sum_{\substack{\col(i)=\col(j)\\\row(i)=a,\row(j)=b}}\gamma_{\col(i)}e_{i,j}[-2]\\
&\quad+\sum_{\substack{\col(x)=\col(j)<\col(i)=\col(v)\\u_1-u_s<\row(x)=\row(v)\leq u_1 \\\row(i)=a,\row(j)=b}}\limits (-1)^{p(x)+p(e_{i,v})p(e_{x,j})}e_{x,j}[-1]e_{i,v}[-1]\\
&\quad+\sum_{\substack{\col(x)=\col(j)<\col(i)=\col(v)\\-q_1\leq\row(x)=\row(v)< -q_1+q_s\\\row(i)=a,\row(j)=b}}\limits (-1)^{p(x)+p(e_{i,v})p(e_{x,j})}e_{x,j}[-1]e_{i,v}[-1]\\
&\quad-\sum_{\substack{\col(x)=\col(j)\geq\col(i)=\col(v)\\q_s-q_1\leq\row(x)=\row(v)\leq q_{\col(j)}-q_1\\\row(i)=a,\row(j)=b}}\limits (-1)^{p(x)+p(e_{i,v})p(e_{x,j})}e_{x,j}[-1]e_{i,v}[-1]\\
&\quad-\sum_{\substack{\col(x)=\col(j)\geq\col(i)=\col(v)\\u_1-u_{\col(j)}\leq\row(x)=\row(v)\leq u_1-u_s\\\row(i)=a,\row(j)=b}}\limits (-1)^{p(x)+p(e_{i,v})p(e_{x,j})}e_{x,j}[-1]e_{i,v}[-1]
\end{align*}
for $a,b\in I_s\setminus I_{s+1}$.
Then, the $W$-superalgebra $\mathcal{W}^k(\mathfrak{gl}(M|N),f)$ contains $W^{(1)}_{a,b}$ and $W^{(2)}_{a,b}$.
\end{Theorem}
\begin{proof}
It is enough to show that $d_0(W^{(r)}_{a,b})=0$ for $r=1,2$ and $a,b\in I_s\setminus I_{s+1}$. We only show the case that $r=2$. The case that $r=1$ can be proven by the formula \eqref{ee307}.

By \eqref{ee1}, if $\col(i)=\col(j)$, we obtain
\begin{align}
[d_0,e_{i,j}[-1]]
&=(-1)^{p(i)}\psi_{\widehat{i},j}[-1]-(-1)^{p(i)}\psi_{i,\widetilde{j}}[-1].\label{ee307}
\end{align}
If $\col(i)=\col(j)+1$, by \eqref{ee1}, we also have
\begin{align}
&\quad[d_0,e_{i,j}[-1]]\nonumber\\
&=\sum_{\substack{\col(r)=\col(j)}}\limits (-1)^{p(e_{i,j})+p(e_{i,r})p(e_{r,j})}e_{r,j}[-1]\psi_{i,r}[-1]\nonumber\\
&\quad-\sum_{\substack{\col(r)=\col(i)}}\limits (-1)^{p(e_{i,r})p(e_{r,j})}\psi_{r,j}[-1]e_{i,r}[-1]\nonumber\\
&\quad+\alpha_{\col(i)}(-1)^{p(i)}\psi_{i,j}[-2]+(-1)^{p(i)}\psi_{\hat{i},j}[-1]-(-1)^{p(i)}\psi_{i,\tilde{j}}[-1].\label{ee2}
\end{align}
By the definition of $W^{(2)}_{i,j}$, we can rewrite $d_0(W^{(2)}_{p,q})$ as
\begin{align}
&\sum_{\substack{\col(i)=\col(j)+1\\\row(i)=a,\row(j)=b}}d_0(e_{i,j}[-1])-\sum_{\substack{\col(i)=\col(j)\\\row(i)=a,\row(j)=b}}\gamma_{\col(i)}d_0(e_{i,j}[-2])\nonumber\\
&\quad+\sum_{\substack{\col(x)=\col(j)<\col(i)=\col(v)\\\row(x)=\row(v)>u_1-u_s\\\row(i)=a,\row(j)=b}}\limits (-1)^{p(x)+p(e_{i,v})p(e_{x,j})}d_0(e_{x,j}[-1])e_{i,v}[-1]\nonumber\\
&\quad+\sum_{\substack{\col(x)=\col(j)<\col(i)=\col(v)\\\row(x)=\row(v)>u_1-u_s\\\row(i)=a,\row(j)=b}}\limits (-1)^{p(j)+p(e_{i,v})p(e_{x,j})}e_{x,j}[-1]d_0(e_{i,v}[-1])\nonumber\\
&\quad+\sum_{\substack{\col(x)=\col(j)<\col(i)=\col(v)\\\row(x)=\row(v)<-q_1+q_s\\\row(i)=a,\row(j)=b}}\limits (-1)^{p(x)+p(e_{i,v})p(e_{x,j})}d_0(e_{x,j}[-1])e_{i,v}[-1]\nonumber\\
&\quad+\sum_{\substack{\col(x)=\col(j)<\col(i)=\col(v)\\\row(x)=\row(v)<-q_1+q_s\\\row(i)=a,\row(j)=b}}\limits (-1)^{p(j)+p(e_{i,v})p(e_{x,j})}e_{x,j}[-1]d_0(e_{i,v}[-1])\nonumber\\
&\quad-\sum_{\substack{\col(x)=\col(j)\geq\col(i)=\col(v)\\q_s-q_1\leq\row(x)=\row(v)\leq q_{\col(j)}-q_1\\\row(i)=a,\row(j)=b}}\limits (-1)^{p(x)+p(e_{i,v})p(e_{x,j})}d_0(e_{x,j}[-1])e_{i,v}[-1]\nonumber\\
&\quad-\sum_{\substack{\col(x)=\col(j)\geq\col(i)=\col(v)\\q_s-q_1\leq\row(x)=\row(v)\leq q_{\col(j)}-q_1\\\row(i)=a,\row(j)=b}}\limits (-1)^{p(j)+p(e_{i,v})p(e_{x,j})}e_{x,j}[-1]d_0(e_{i,v}[-1])\nonumber\\
&\quad-\sum_{\substack{\col(x)=\col(j)\geq\col(i)=\col(v)\\u_1-u_{\col(j)}\leq\row(x)=\row(v)\leq u_1-u_s\\\row(i)=a,\row(j)=b}}\limits (-1)^{p(x)+p(e_{i,v})p(e_{x,j})}d_0(e_{x,j}[-1])e_{i,v}[-1]\nonumber\\
&\quad-\sum_{\substack{\col(x)=\col(j)\geq\col(i)=\col(v)\\u_1-u_{\col(j)}\leq\row(x)=\row(v)\leq u_1-u_s\\\row(i)=a,\row(j)=b}}\limits (-1)^{p(j)+p(e_{i,v})p(e_{x,j})}e_{x,j}[-1]d_0(e_{i,v}[-1]).\label{ee3}
\end{align}
By \eqref{ee2}, we obtain
\begin{align}
&\quad\eqref{ee3}_1\nonumber\\
&=\sum_{\substack{\col(i)=\col(j)+1\\\row(i)=a,\row(j)=b}}\sum_{\substack{\col(r)=\col(j)}}\limits (-1)^{p(e_{i,j})+p(e_{i,r})p(e_{r,j})}e_{r,j}[-1]\psi_{i,r}[-1]\nonumber\\
&\quad-\sum_{\substack{\col(i)=\col(j)+1\\\row(i)=a,\row(j)=b}}\sum_{\substack{\col(r)=\col(i)}}\limits (-1)^{p(e_{i,r})p(e_{r,j})}\psi_{r,j}[-1]e_{i,r}[-1]\nonumber\\
&\quad+\sum_{\substack{\col(i)=\col(j)+1\\\row(i)=a,\row(j)=b}}(-1)^{p(i)}\alpha_{\col(i)}\psi_{i,j}[-2]+\sum_{\substack{\col(i)=\col(j)+1\\\row(i)=a,\row(j)=b}}(-1)^{p(i)}(\psi_{\hat{i},j}[-1]-\psi_{i,\tilde{j}}[-1]).\label{ee5}
\end{align}
By a direct computation, the last term of the right hand side of \eqref{ee5} is equal to zero. Then, we have
\begin{align}
\eqref{ee3}_1
&=\sum_{\substack{\col(i)=\col(j)+1\\\row(i)=a,\row(j)=b}}\sum_{\substack{\col(r)=\col(j)}}\limits (-1)^{p(e_{i,j})+p(e_{i,r})p(e_{r,j})}e_{r,j}[-1]\psi_{i,r}[-1]\nonumber\\
&\quad-\sum_{\substack{\col(i)=\col(j)+1\\\row(i)=a,\row(j)=b}}\sum_{\substack{\col(r)=\col(i)}}\limits (-1)^{p(e_{i,r})p(e_{r,j})}\psi_{r,j}[-1]e_{i,r}[-1]\nonumber\\
&\quad+\sum_{\substack{\col(i)=\col(j)+1\\\row(i)=a,\row(j)=b}}(-1)^{p(i)}\alpha_{\col(i)}\psi_{i,j}[-2].\label{ee5.1}
\end{align}

By \eqref{ee307} and \eqref{ee5800}, we obtain
\begin{align}
\eqref{ee3}_2
&=-\sum_{\substack{\col(i)=\col(j)\\\row(i)=a,\row(j)=b}}(-1)^{p(i)}\gamma_{\col(i)}(\psi_{\widehat{i},j}[-2]-\psi_{i,\widetilde{j}}[-2])\nonumber\\
&=\sum_{\substack{\col(i)=\col(j)\\\row(i)=a,\row(j)=b}}(-1)^{p(i)}(\gamma_{\col(\hat{i})}-\gamma_{\col(i)})\psi_{\widehat{i},j}[-2]\nonumber\\
&=-\sum_{\substack{\col(i)=\col(j)\\\row(i)=a,\row(j)=b}}\limits(-1)^{p(i)}\alpha_{\col(\hat{i})}\psi_{\hat{i},j}[-2].\label{ee6}
\end{align}
By \eqref{ee307}, we obtain
\begin{align}
&\quad\eqref{ee3}_3\nonumber\\
&=\sum_{\substack{\col(x)=\col(j)<\col(i)=\col(v)\\\row(x)=\row(v)>u_1-u_s\\\row(i)=a,\row(j)=b}}\limits (-1)^{p(x)+p(e_{i,v})p(e_{x,j})+p(x)}(\psi_{\hat{x},j}[-1]-\psi_{x,\tilde{j}}[-1])e_{i,v}[-1]\nonumber\\
&=\sum_{\substack{\col(x)+1=\col(j)+1=\col(i)=\col(v)\\\row(x)=\row(v)>u_1-u_s\\\row(i)=a,\row(j)=b}}\limits (-1)^{p(e_{i,v})p(e_{x,j})}\psi_{\hat{x},j}[-1]e_{i,v}[-1]\label{ee7},\\
&\quad\eqref{ee3}_4\nonumber\\
&=\sum_{\substack{\col(x)=\col(j)<\col(i)=\col(v)\\\row(x)=\row(v)>u_1-u_s\\\row(i)=a,\row(j)=b}}\limits (-1)^{p(j)+p(e_{i,v})p(e_{x,j})+p(i)}e_{x,j}[-1](\psi_{\hat{i},v}[-1]-\psi_{i,\tilde{v}}[-1])\nonumber\\
&=-\sum_{\substack{\col(x)+1=\col(j)+1=\col(i)=\col(v)\\\row(x)=\row(v)>u_1-u_s\\\row(i)=a,\row(j)=b}}\limits (-1)^{p(i)+p(j)+p(e_{i,v})p(e_{x,j})}e_{x,j}[-1]\psi_{i,\tilde{v}}[-1],\label{ee8}\\
&\quad\eqref{ee3}_5\nonumber\\
&=\sum_{\substack{\col(x)=\col(j)<\col(i)=\col(v)\\\row(x)=\row(v)<-q_1+q_s\\\row(i)=a,\row(j)=b}}\limits (-1)^{p(x)+p(e_{i,v})p(e_{x,j})+p(x)}(\psi_{\hat{x},j}[-1]-\psi_{x,\tilde{j}}[-1])e_{i,v}[-1]\nonumber\\
&=\sum_{\substack{\col(x)+1=\col(j)+1=\col(i)=\col(v)\\\row(x)=\row(v)<-q_1+q_s\\\row(i)=a,\row(j)=b}}\limits (-1)^{p(e_{i,v})p(e_{x,j})}\psi_{\hat{x},j}[-1]e_{i,v}[-1]\label{ee7-1},\\
&\quad\eqref{ee3}_6\nonumber\\
&=\sum_{\substack{\col(x)=\col(j)<\col(i)=\col(v)\\\row(x)=\row(v)< -q_1+q_s\\\row(i)=a,\row(j)=b}}\limits (-1)^{p(j)+p(e_{i,v})p(e_{x,j})+p(i)}e_{x,j}[-1](\psi_{\hat{i},v}[-1]-\psi_{i,\tilde{v}}[-1])\nonumber\\
&=-\sum_{\substack{\col(x)+1=\col(j)+1=\col(i)=\col(v)\\\row(x)=\row(v)<-q_1+q_s\\\row(i)=a,\row(j)=b}}\limits (-1)^{p(i)+p(j)+p(e_{i,v})p(e_{x,j})}e_{x,j}[-1]\psi_{i,\tilde{v}}[-1],\label{ee8-1}\\
&\quad\eqref{ee3}_7\nonumber\\
&=-\sum_{\substack{\col(x)=\col(j)\geq\col(i)=\col(v)\\q_s-q_1\leq\row(x)=\row(v)\leq q_{\col(j)}-q_1\\\row(i)=a,\row(j)=b}}\limits (-1)^{p(x)+p(e_{i,v})p(e_{x,j})+p(x)}(\psi_{\hat{x},j}[-1]-\psi_{u,\tilde{j}}[-1])e_{i,v}[-1]\nonumber\\
&=\sum_{\substack{\col(x)=\col(j)=\col(i)+1=\col(v)+1\\q_s-q_1\leq\row(x)=\row(v)\leq q_{\col(j)}-q_1\\\row(i)=a,\row(j)=b}}\limits(-1)^{p(e_{i,v})p(e_{x,j})}\psi_{x,\tilde{j}}[-1]e_{i,v}[-1],\label{ee9}\\
&\quad\eqref{ee3}_8\nonumber\\
&=-\sum_{\substack{\col(x)=\col(j)\geq\col(i)=\col(v)\\q_s-q_1\leq\row(x)=\row(v)\leq q_{\col(j)}-q_1\\\row(i)=a,\row(j)=b}}\limits(-1)^{p(j)+p(e_{i,v})p(e_{x,j})+p(i)} e_{x,j}[-1](\psi_{\hat{i},v}[-1]-\psi_{i,\tilde{v}}[-1])\nonumber\\
&=-\sum_{\substack{\col(x)=\col(j)=\col(i)+1=\col(v)+1\\q_s-q_1\leq\row(x)=\row(v)\leq q_{\col(j)}-q_1\\\row(i)=a,\row(j)=b}}\limits(-1)^{p(i)+p(j)+p(e_{i,v})p(e_{x,j})} e_{x,j}[-1]\psi_{\hat{i},v}[-1],\label{ee10}\\
&\quad\eqref{ee3}_9\nonumber\\
&=-\sum_{\substack{\col(x)=\col(j)\geq\col(i)=\col(v)\\u_1-u_{\col(j)}\leq\row(x)=\row(v)\leq u_1-u_s\\\row(i)=a,\row(j)=b}}\limits (-1)^{p(x)+p(e_{i,v})p(e_{x,j})+p(x)}(\psi_{\hat{x},j}[-1]-\psi_{u,\tilde{j}}[-1])e_{i,v}[-1]\nonumber\\
&=\sum_{\substack{\col(x)=\col(j)=\col(i)+1=\col(v)+1\\u_1-u_{\col(j)}\leq\row(u)=\row(v)\leq u_1-u_s\\\row(i)=a,\row(j)=b}}\limits(-1)^{p(e_{i,v})p(e_{x,j})}\psi_{x,\tilde{j}}[-1]e_{i,v}[-1],\label{ee11}\\
&\quad\eqref{ee3}_{10}\nonumber\\
&=-\sum_{\substack{\col(x)=\col(j)\geq\col(i)=\col(v)\\u_1-u_{\col(j)}\leq\row(x)=\row(v)\leq u_1-u_s\\\row(i)=a,\row(j)=b}}\limits(-1)^{p(j)+p(e_{i,v})p(e_{x,j})+p(i)} e_{x,j}[-1](\psi_{\hat{i},v}[-1]-\psi_{i,\tilde{v}}[-1])\nonumber\\
&=-\sum_{\substack{\col(x)=\col(j)=\col(i)+1=\col(v)+1\\u_1-u_{\col(j)}\leq\row(x)=\row(v)\leq u_1-u_s\\\row(i)=a,\row(j)=b}}\limits(-1)^{p(i)+p(j)+p(e_{i,v})p(e_{x,j})} e_{x,j}[-1]\psi_{\hat{i},v}[-1].\label{ee12}
\end{align}
By a direct computation, we obtain
\begin{gather*}
\eqref{ee5}_1+\eqref{ee8}+\eqref{ee8-1}+\eqref{ee10}+\eqref{ee12}=0,\\
\eqref{ee5}_2+\eqref{ee7}+\eqref{ee7-1}+\eqref{ee9}+\eqref{ee11}=0,\\
\eqref{ee5}_3+\eqref{ee6}=0.
\end{gather*}
Then, by adding \eqref{ee5}-\eqref{ee10}, we obtain $d_0(W^{(2)}_{i,j})=0$.
\end{proof}
Then, by Theorem 5.2 in \cite{Genra} and Theorem 14 in \cite{Nak}, there exists an embedding
\begin{equation*}
\mu\colon\mathcal{W}^k(\mathfrak{gl}(M|N),f)\to\bigotimes_{1\leq s\leq l}V^{\kappa_s}(\mathfrak{gl}(u_s|q_s)).
\end{equation*}
This embedding is called the Miura map. 
\begin{Theorem}\label{Generators}
For integers $a,b\in I_s\setminus I_{s+1}$, the following elements of $\bigotimes_{1\leq s\leq l}\limits V^{\kappa_s}(\mathfrak{gl}(q_s))$ are contained in $\mu(\mathcal{W}^k(\mathfrak{gl}(M|N),f))$:
\begin{align*}
\mu(W^{(1)}_{a,b})&=\sum_{\substack{\row(i)=a,\row(j)=b,\\\col(i)=\col(j)}}e_{i,j}[-1],\\
\mu(W^{(2)}_{a,b})&=-\sum_{\substack{\col(i)=\col(j)\\\row(i)=a,\row(j)=b}}\gamma_{\col(i)}e_{i,j}[-2]\\
&\quad+\sum_{\substack{\col(x)=\col(j)<\col(i)=\col(v)\\\row(x)=\row(v)> u_1-u_s\\\row(i)=a,\row(j)=b}}\limits (-1)^{p(x)+p(e_{i,v})p(e_{x,j})}e_{x,j}[-1]e_{i,v}[-1]\\
&\quad+\sum_{\substack{\col(x)=\col(j)<\col(i)=\col(v)\\\row(x)=\row(v)< -q_1+q_s\\\row(i)=a,\row(j)=b}}\limits (-1)^{p(x)+p(e_{i,v})p(e_{x,j})}e_{x,j}[-1]e_{i,v}[-1]\\
&\quad-\sum_{\substack{\col(x)=\col(j)\geq\col(i)=\col(v)\\q_s-q_1\leq\row(x)=\row(v)\leq q_{\col(j)}-q_1\\\row(i)=a,\row(j)=b}}\limits (-1)^{p(x)+p(e_{i,v})p(e_{x,j})}e_{x,j}[-1]e_{i,v}[-1]\\
&\quad-\sum_{\substack{\col(x)=\col(j)\geq\col(i)=\col(v)\\u_1-u_{\col(j)}\leq\row(x)=\row(v)\leq u_1-u_s\\\row(i)=a,\row(j)=b}}\limits (-1)^{p(x)+p(e_{i,v})p(e_{x,j})}e_{x,j}[-1]e_{i,v}[-1].
\end{align*}
\end{Theorem}
Let us recall the definition of the universal enveloping algebras of vertex superalgebras.
For any vertex superalgebra $V$, let $L(V)$ be the Borchards Lie algebra, that is,
\begin{align}
 L(V)=V{\otimes}\mathbb{C}[t,t^{-1}]/\text{Im}(\partial\otimes\id +\id\otimes\frac{d}{d t})\label{844},
\end{align}
where the commutation relation is given by
\begin{align*}
 [ut^a,vt^b]=\sum_{r\geq 0}\begin{pmatrix} a\\r\end{pmatrix}(u_{(r)}v)t^{a+b-r}
\end{align*}
for all $u,v\in V$ and $a,b\in \mathbb{Z}$. 
\begin{Definition}[Frenkel-Zhu~\cite{FZ}, Matsuo-Nagatomo-Tsuchiya~\cite{MNT}]\label{Defi}
We set $\mathcal{U}(V)$ as the quotient algebra of the standard degreewise completion of the universal enveloping algebra of $L(V)$ by the completion of the two-sided ideal generated by
\begin{gather}
(u_{(a)}v)t^b-\sum_{i\geq 0}
\begin{pmatrix}
 a\\i
\end{pmatrix}
(-1)^i(ut^{a-i}vt^{b+i}-{(-1)}^{p(u)p(v)}(-1)^avt^{a+b-i}ut^{i}),\label{241}\\
|0\rangle t^{-1}-1,
\end{gather}
where $|0\rangle$ is the identity vector of $V$.
We call $\mathcal{U}(V)$ the universal enveloping algebra of $V$.
\end{Definition}
By the definition of the universal affine vertex algebra $V^{\kappa}(\mathfrak{g})$ associated with a finite dimensional reductive Lie algebra $\mathfrak{g}$ and the inner product $\kappa$ on $\mathfrak{g}$, $\mathcal{U}(V^\kappa(\mathfrak{g}))$ is the standard degreewise completion of the universal enveloping algebra of the affinization of $\mathfrak{g}$.

Then, induced by the Miura map $\mu$, we obtain the embedding
\begin{equation*}
\widetilde{\mu}\colon \mathcal{U}(\mathcal{W}^{k}(\mathfrak{gl}(M|N),f))\to {\widehat{\bigotimes}}_{1\leq a\leq l}U(\widehat{\mathfrak{gl}}(u_a|q_a)),
\end{equation*}
where ${\widehat{\bigotimes}}_{1\leq a\leq l}U(\widehat{\mathfrak{gl}}(u_a|q_a))$ is the standard degreewise completion of $\bigotimes_{1\leq a\leq l}U(\widehat{\mathfrak{gl}}(u_a|q_a))$. Here after, we embed $U(\widehat{\mathfrak{gl}}(u_a|q_a))$ into $U(\widehat{\mathfrak{gl}}(u_1|q_1))$ by
\begin{gather}
E_{i,j}t^s\mapsto E_{u_1-u_a+i,u_1-u_a+j}\text{ if }i,j>0,\\
E_{i,j}t^s\mapsto E_{-q_1+q_a+i,u_1-u_a+j}\text{ if }i<0,j>0,\\
E_{i,j}t^s\mapsto E_{u_1-u_a+i,-q_1+q_a+j}\text{ if }i>0,j<0\\
E_{i,j}t^s\mapsto E_{-q_1+q_a+i,-q_1+q_a+j}\text{ if }i,j<0.
\end{gather}
We denote the elements
\begin{equation*}
E^{(a)}_{i,j}t^v=1^{\otimes a-1}\otimes E_{i,j}t^v\otimes1^{l-v}\in\bigotimes_{1\leq s\leq l}U(\widehat{\mathfrak{gl}}(u_s|q_s))
\end{equation*}
for $1\leq a\leq l$. Let us assume that $u_s-u_{s+1}\geq2$.
Then, by Theorem~\ref{Generators}, we have
\begin{align}
&\quad\widetilde{\mu}(W^{(2)}_{u_1-u_s+1,u_1-u_s+2}t)\nonumber\\
&=\sum_{1\leq a\leq s}\gamma_{a}E^{(a)}_{u_1-u_s+1,u_1-u_s+2}+\sum_{v\in\mathbb{Z}}\sum_{r_1<r_2}\limits\sum_{x>u_1-u_s}\limits E^{(r_1)}_{x,u_1-u_s+2}t^{-v}E^{(r_2)}_{u_1-u_s+1,x}t^v\nonumber\\
&\quad+\sum_{v\in\mathbb{Z}}\sum_{r_1<r_2}\limits\sum_{x<-q_1+q_s}\limits E^{(r_1)}_{x,u_1-u_s+2}t^{-v}E^{(r_2)}_{u_1-u_s+1,x}t^v\nonumber\\
&\quad-\sum_{v\in\mathbb{Z}}\sum_{r_1>r_2}\limits\sum_{\substack{q_s-q_1\leq x\leq q_{r_1}-q_1-1}}\limits E^{(r_1)}_{x,u_1-u_s+2}t^{-v}E^{(r_2)}_{u_1-u_s+1,x}t^v\nonumber\\
&\quad-\sum_{v\in\mathbb{Z}}\sum_{r_1>r_2}\limits\sum_{\substack{u_1-u_{r_1}+1\leq x\leq u_1-u_s}}\limits E^{(r_1)}_{x,u_1-u_s+2}t^{-v}E^{(r_2)}_{u_1-m_s+1,x}t^v\nonumber\\
&\quad-\sum_{v\geq0}\sum_{r=1}^{s}\limits\sum_{\substack{q_s-q_1\leq x\leq q_{r}-q_1-1}}\limits (E^{(r)}_{x,u_1-u_s+2}t^{-v-1}E^{(r)}_{u_1-u_s+1,x}t^{v+1}-E^{(r)}_{u_1-u_s+1,x}t^{-v}E^{(r)}_{x,u_1-u_s+2}t^{v})\nonumber\\
&\quad-\sum_{v\geq0}\sum_{r=1}^{s}\limits\sum_{\substack{u_1-u_r+1\leq x\leq u_1-u_s}}\limits (E^{(r)}_{x,u_1-u_s+2}t^{-v-1}E^{(r)}_{u_1-u_s+1,x}t^{v+1}+E^{(r)}_{u_1-u_s+1,x}t^{-v}E^{(r)}_{x,u_1-u_s+2}t^{v}).\label{W-gen}
\end{align}
By \eqref{W-gen}, $\widetilde{\mu}(W^{(2)}_{u_1-u_s+1,u_1-u_s+2}t)$ can be regarded as an element of $\widehat{\bigotimes}_{1\leq i\leq s}U(\widehat{\mathfrak{gl}}(u_i|q_i))$.
\section{Affine super Yangians and $W$-superalgebras}
In non-super setting, we \cite{U11} constructed a homomorphism from the affine Yangian to the universal enveloping algebra of a non-rectangular $W$-algebra of type $A$ by using the edge contractions of the affine Yangian. Similarly to \cite{U11}, we will give a homomorphism from the affine super Yangian to the universal enveloping algebra of a non-rectangular $W$-superalgebra of type $A$.
 In this section, we do not identify $I$ with $\mathbb{Z}/(m+n)\mathbb{Z}$.
\begin{Theorem}
\begin{enumerate}
\item For $m_2,n_2\geq 0$, $m_1,n_1\geq 2$ and $m_1+n_1\geq 5$, there exists a homomorphism
\begin{gather*}
\Psi_1^{m_1|n_1,m_1+m_2|n_1+n_2}\colon Y_{\hbar,\ve}(\widehat{\mathfrak{sl}}(m_1|n_1))\to \widetilde{Y}_{\hbar,\ve}(\widehat{\mathfrak{sl}}(m_1+m_2|n_1+n_2))
\end{gather*}
given by
\begin{gather*}
\Psi_1^{m_1|n_1,m_1+m_2|n_1+n_2}(X^+_{i,0})=\begin{cases}
E_{-n_1,1}t&\text{ if }i=-n_1,\\
E_{i,i+1}&\text{ if }1\leq i\leq m_1-1,\\
E_{m_1,-1}&\text{ if }i=m_1,\\
E_{i,i-1}&\text{ if }-n_1+1\leq i\leq -1
\end{cases}\\
\Psi_1^{m_1|n_1,m_1+m_2|n_1+n_2}(X^-_{i,0})=\begin{cases}
-E_{1,-n_1}t^{-1}&\text{ if }i=n_1,\\
E_{i+1,i}&\text{ if }1\leq i\leq m_1-1,\\
E_{-1,m_1}&\text{ if }i=m_1,\\
-E_{i-1,i}&\text{ if }-n_1+1\leq i\leq -1
\end{cases}
\end{gather*}
and
\begin{align*}
\Psi_1^{m_1|n_1,m_1+m_2|n_1+n_2}(X^+_{1,1})
&= X^+_{1,1}-\hbar\displaystyle\sum_{v\geq0} \limits\sum_{z=m_1+1}^{m_1+m_2}\limits E_{1,z}t^{-v-1} E_{z,2}t^{v+1}\\
&\quad+\hbar\displaystyle\sum_{v\geq0} \limits\sum_{z=-n_1-n_2}^{-n_1-1}\limits E_{1,z}t^{-v-1} E_{z,2}t^{v+1}.
\end{align*}
\item For $m_1,n_1\geq 0$, $m_2,n_2\geq 2$ and $m_2+n_2\geq5$, there exists a homomorphism
\begin{gather*}
\Psi_2^{m_2|n_2,m_1+m_2|n_1+n_2}\colon Y_{\hbar,\ve+(m_1-n_1)\hbar}(\widehat{\mathfrak{sl}}(m_2|n_2))\to \widetilde{Y}_{\hbar,\ve}(\widehat{\mathfrak{sl}}(m_1+m_2|n_1+n_2))
\end{gather*}
determined by
\begin{gather*}
\Psi_2^{m_2|n_2,m_1+m_2|n_1+n_2}(X^+_{i,0})=\begin{cases}
E_{-n_1-n_2,m_1+1}t&\text{ if }i=-n_2,\\
E_{m_1+i,m_1+i+1}&\text{ if }1\leq i\leq m_2-1,\\
E_{m_1+m_2,-n_1-1}&\text{ if }i=m_2,\\
E_{-n_1+i,-n_1+i-1}&\text{ if }-n_2+1\leq i\leq -1,
\end{cases}\\ 
\Psi_2^{m_2|n_2,m_1+m_2|n_1+n_2}(X^-_{i,0})=\begin{cases}
-E_{m_1+1,-n_1-n_2}t^{-1}&\text{ if }i=-n_2,\\
E_{m_1+i+1,m_1+i}&\text{ if }1\leq i\leq m_2-1,\\
E_{-n_1-1,m_1+m_2}&\text{ if }i=m_2,\\
-E_{-n_1+i-1,-n_1+i}&\text{ if }-n_2+1\leq i\leq -1,
\end{cases}
\end{gather*}
and
\begin{align*}
\Psi_2^{m_2|n_2,m_1+m_2|n_1+n_2}(X^+_{1,1})
&= X^+_{1+m_1,1}+\hbar\displaystyle\sum_{v\geq0}\limits\sum_{z=-n_1}^{-1} E_{z,2+m_1}t^{-v}E_{1+m_1,z}t^{s}\\
&\quad+\hbar\displaystyle\sum_{v\geq0}\limits\sum_{z=1}^{m_1} E_{z,2+m_1}t^{-v-1} E_{1+m_1,z}t^{v+1}.
\end{align*}
\end{enumerate}
\end{Theorem}
For $1\leq s\leq l$, we define $\ve_s=\hbar(k+N-u_s+q_s)$. In the case that $u_s-u_{s+1},q_s-q_{s+1}\geq 2$ and $u_s-q_s+q_s-q_{s+1}\geq5$, let us define the homomorphism
\begin{equation*}
\Delta^{s}\colon Y_{\hbar,\ve_s}(\widehat{\mathfrak{sl}}(u_s-u_{s+1}|q_s-q_{s+1}))\to\bigotimes_{1\leq a\leq s}Y_{\hbar,\ve_a}(\widehat{\mathfrak{sl}}(u_a|q_a))
\end{equation*}
defined by
\begin{equation*}
\Delta^{s}=(\prod_{a=1}^{s-1}\limits (((\Psi_2^{q_{a+1}|u_{a+1},q_a|u_a}\otimes1)\circ\Delta)\otimes\id^{s-a-1})\circ\Psi_1^{u_s-u_{s+1}|q_s-q_{s+1},z_s|q_s})\otimes\id^{l-s}.
\end{equation*}
\begin{Theorem}\label{Main3}
Assume that $\ve_s=\hbar(k+M-N-u_s+q_s)$. There exists an algebra homomorphism 
\begin{equation*}
\Phi\colon Y_{\hbar,\ve_s}(\widehat{\mathfrak{sl}}(u_s-u_{s+1}|q_s-q_{s+1}))\to \mathcal{U}(\mathcal{W}^{k}(\mathfrak{gl}(N),f))
\end{equation*} 
determined by
\begin{equation}
\bigotimes_{1\leq a\leq s}\ev_{\hbar,\ve_a}^{u_a|q_a,-x_a\hbar}\circ\Delta^{s}=\widetilde{\mu}\circ\Phi,\label{eqqq}
\end{equation}
where $x_a=\gamma_a+q_a-q_s-\dfrac{u_a-u_s}{2}$.
\end{Theorem}
\begin{proof}
Let us set
\begin{gather*}
\Phi(X^+_{i,0})=\begin{cases}
W^{(1)}_{u_1-u_s+i,u_1-u_s+i+1}&\text{ if }1\leq i\leq u_s-u_{s+1}-1,\\
W^{(1)}_{u_1-u_s+i,-q_1+q_s-1}&\text{ if }i=u_s-u_{s+1},\\
W^{(1)}_{u_1-u_s+i,-q_1+q_s+i-1}&\text{ if }-q_s+q_{s+1}+1\leq i\leq-1,\\
W^{(1)}_{-q_1+q_{s+1},u_1-u_s+1}t&\text{ if }i=-q_s+q_{s+1},
\end{cases}\\
\Phi(X^-_{i,0})=\begin{cases}
W^{(1)}_{u_1-u_s+i+1,u_1-u_s+i}&\text{ if }1\leq i\leq u_s-u_{s+1}-1,\\
W^{(1)}_{-q_1+q_s-1,u_1-u_s+i}&\text{ if }i=u_s-u_{s+1},\\
W^{(1)}_{-q_1+q_s+i-1,u_1-u_s+i}&\text{ if }-q_s+q_{s+1}+1\leq i\leq-1,\\
W^{(1)}_{u_1-u_s+1,-q_1+q_{s+1}}t^{-1}&\text{ if }i=-q_s+q_{s+1},
\end{cases}
\end{gather*}
and
\begin{align*}
\Phi(X^+_{1,1})&=
-\hbar W^{(2)}_{u_1-u_s+1,u_1-u_s+2}t-\dfrac{1}{2}\hbar W^{(1)}_{u_1-u_s+1,u_1-u_s+2}\\
&\quad+\hbar\displaystyle\sum_{v\geq0}\limits W^{(1)}_{u_1-u_s+1,u_1-u_s+1}t^{-v}W^{(1)}_{u_1-u_s+1,u_1-u_s+2}t^v\\
&\quad+\hbar\displaystyle\sum_{s\geq 0} \limits\displaystyle\sum_{z=2}^{u_s-u_{s+1}}\limits W^{(1)}_{u_1-u_s+1,u_1-u_s+z}t^{-v-1} W^{(1)}_{u_1-u_s+z,u_1-u_s+2}t^{v+1}\\
&\quad-\hbar\displaystyle\sum_{v\geq0}\limits\displaystyle\sum_{z=-q_s+q_{s+1}}^{-1} \limits W^{(1)}_{-q_1+q_s+1,-q_1+q_s+z}t^{-v-1} W^{(1)}_{-q_1+q_s+z,-q_1+q_s+2}t^{v+1}.
\end{align*}
It is enough to show the relation
\begin{equation}
\bigotimes_{1\leq a\leq s}\ev_{\hbar,\ve_a}^{u_a|q_a,-x_a\hbar}\circ\Delta^{s}(X^\pm_{i,0})=\widetilde{\mu}\circ\Phi(X^\pm_{i,0})\label{eqqq1}
\end{equation}
for $1\leq i\leq u_s-u_{s+1}$ and $-q_s+q_{s+1}\leq i\leq -1$ and
\begin{equation}
\bigotimes_{1\leq a\leq s}\ev_{\hbar,\ve_a}^{u_a|q_a,-x_a\hbar}\circ\Delta^{s}(X^+_{1,1})=\widetilde{\mu}\circ\Phi(X^+_{1,1}).\label{eqqq2}
\end{equation}
The relation \eqref{eqqq1} is obvious so that we only show the relation \eqref{eqqq2}. By the definition of edge contractions and coproduct for the affine super Yangian, we obtain
\begin{align*}
\bigotimes_{1\leq a\leq s}\ev_{\hbar,\ve_a}^{u_a|q_a,-x_a\hbar}\circ\Delta^s(X^+_{1,1})&=\sum_{1\leq a\leq s}\ev_{\hbar,\ve_a}^{u_a|q_a,-x_a\hbar}(X^+_{1+u_a-u_s,1})^{(a)}+B+C+D,
\end{align*}
where
\begin{align}
B&=\hbar\displaystyle\sum_{v\geq0} \limits \sum_{r_1<r_2}\limits\displaystyle\sum_{z=1+u_{r_2}-u_s}^{1+u_1-u_s}\limits E_{1+u_1-u_s,z}^{(r_1)}t^{-v}E_{z,2+u_1-u_s}^{(r_2)}t^v\nonumber\\
&\quad-\hbar\displaystyle\sum_{v\geq0} \limits \sum_{r_1<r_2}\limits\displaystyle\sum_{z=1+u_{r_2}-u_s}^{1+u_1-u_s}\limits E_{z,2+u_1-u_s}^{(r_1)}t^{-v-1}E_{1+u_1-u_s,z}^{(r_2)}t^{v+1}\nonumber\\
&\quad+\hbar\sum_{a=1}^{s}\displaystyle\sum_{v\geq0} \limits \sum_{r_1<r_2}\limits\displaystyle\sum_{z=2+u_1-u_s}^{u_1}\limits E_{1+u_1-u_s,z}^{(r_1)}t^{-v-1}E_{z,2+u_1-u_s}^{(r_2)}t^{v+1}\nonumber\\
&\quad-\hbar\sum_{a=1}^{s}\displaystyle\sum_{v\geq0} \limits \sum_{r_1<r_2}\limits\displaystyle\sum_{z=2+u_1-u_s}^{u_1}\limits E_{z,2+u_1-u_s}^{(r_1)}t^{-v}E_{1+u_1-u_s,z}^{(r_2)}t^{s}\nonumber\\
&\quad-\hbar\displaystyle\sum_{v\geq0} \limits \sum_{r_1<r_2}\limits\displaystyle\sum_{z=-q_1}^{-1-q_1+q_{r_2}}\limits E_{1+u_1-u_s,z}^{(r_1)}t^{-v-1}E_{z,2+u_1-u_s}^{(r_2)}t^{v+1}\nonumber\\
&\quad-\hbar\displaystyle\sum_{v\geq0} \limits \sum_{r_1<r_2}\limits\displaystyle\sum_{z=-q_1}^{-1-q_1+q_{r_2}}\limits E_{z,2+u_1-u_s}^{(r_1)}t^{-v}E_{1+u_1-u_s,z}^{(r_2)}t^{s},\label{B}\\
C&=-\hbar\sum_{r=1}^{s}\limits\displaystyle\sum_{v\geq0} \limits\sum_{x=u_1-u_{s+1}+1}^{u_1}\limits E^{(r)}_{1+u_1-u_s,x}t^{-v-1} E^{(r)}_{x,2+u_1-u_s}t^{v+1}\nonumber\\
&\quad-\hbar\sum_{r_1<r_2}\limits\displaystyle\sum_{v\geq0} \limits\sum_{x=u_1-u_{s+1}+1}^{u_1}\limits E^{(r_1)}_{1+u_1-u_s,x}t^{-v-1} E^{(r_2)}_{x,2+u_1-u_s}t^{v+1}\nonumber\\
&\quad-\hbar\sum_{r_1<r_2}\limits\displaystyle\sum_{v\geq0} \limits\sum_{x=u_1-u_{s+1}+1}^{u_1}\limits E^{(r_2)}_{1+u_1-u_s,x}t^{-v-1} E^{(r_1)}_{x,2+u_1-u_s}t^{v+1}\nonumber\\
&\quad+\hbar\sum_{r=1}^l\limits\displaystyle\sum_{v\geq0} \limits\sum_{x=-q_1}^{-q_1+q_{s+1}-1}\limits E^{(r)}_{1+u_1-u_s,x}t^{-v-1} E^{(r)}_{x,2+u_1-u_s}t^{v+1}\nonumber\\
&\quad+\hbar\sum_{r_1<r_2}\limits\displaystyle\sum_{v\geq0} \limits\sum_{x=-q_1}^{-q_1+q_{s+1}-1}\limits E^{(r_1)}_{1+u_1-u_s,x}t^{-v-1} E^{(r_2)}_{x,2+u_1-u_s}t^{v+1}\nonumber\\
&\quad+\hbar\sum_{r_1<r_2}\limits\displaystyle\sum_{v\geq0} \limits\sum_{x=-q_1}^{-q_1+q_{s+1}-1}\limits E^{(r_2)}_{1+u_1-u_s,x}t^{-v-1} E^{(r_1)}_{x,2+u_1-u_s}t^{v+1},\label{C}\\
D&=\hbar\sum_{r=1}^{s}\displaystyle\sum_{v\geq0}\limits\sum_{x=-q_1+q_s}^{-1-q_1+q_{r_2}} E^{(r)}_{x,2+u_1-u_s}t^{-v}E^{(r)}_{1+u_1-u_s,x}t^{s}\nonumber\\
&\quad+\hbar\sum_{r_1<r_2}\limits\displaystyle\sum_{v\geq0}\limits\sum_{x=-q_1+q_s}^{-1-q_1+q_{r_2}} E^{(r_2)}_{x,2+u_1-u_s}t^{-v}E^{(r_1)}_{1+u_1-u_s,x}t^{s}\nonumber\\
&\quad+\hbar\sum_{r_1<r_2}\limits\displaystyle\sum_{v\geq0}\limits\sum_{x=-q_1+q_s}^{-1-q_1+q_{r_2}} E^{(r_1)}_{x,2+u_1-u_s}t^{-v}E^{(r_2)}_{1+u_1-u_s,x}t^{s}\nonumber\\
&\quad+\hbar\sum_{r=1}^{s}\displaystyle\sum_{v\geq0}\limits\sum_{x=u_1-u_r+1}^{u_1-u_s} E^{(r)}_{x,2+u_1-u_s}t^{-v-1} E^{(r)}_{1+u_1-u_s,x}t^{v+1}\nonumber\\
&\quad+\hbar\sum_{r_1<r_2}\displaystyle\sum_{v\geq0}\limits\sum_{x=u_1-u_{r_2}+1}^{u_1-u_s} E^{(r_1)}_{x,2+u_1-u_s}t^{-v-1} E^{(r_2)}_{1+u_1-u_s,x}t^{v+1}\nonumber\\
&\quad+\hbar\sum_{r_2>r_1}\displaystyle\sum_{v\geq0}\limits\sum_{x=u_1-u_{r_1}+1}^{u_1-u_s} E^{(r_1)}_{x,2+u_1-u_s}t^{-v-1} E^{(r_2)}_{1+u_1-u_s,x}t^{v+1}.\label{D}
\end{align}
We note that $B$, $C$ and $D$ come from the coproduct for the affine Yangian, the homomorphism $\Psi_1^{m_1|n_1,m_1+m_2|n_1+n_2}$ and the homomorphism $\Psi_2^{m_2|n_2,m_1+m_2|n_1+n_2}$ respectively.
By the definition of the evaluation map, we have
\begin{align}
\sum_{1\leq a\leq s}\ev_{\hbar,\ve_a}^{u_a|q_a,-x_a\hbar}(X^+_{1+u_a-u_s,1})^{(a)}
&=-\hbar\sum_{1\leq r\leq s}\dfrac{1+u_r-u_s}{2}E^{(r)}_{1+u_1-u_s,2+u_1-u_s}\nonumber\\
&\quad+\hbar \sum_{r=1}^{s}\displaystyle\sum_{v\geq0}  \limits\displaystyle\sum_{g=1+u_1-u_r}^{1+u_1-u_s}\limits E^{(r)}_{1+u_1-u_s,g}t^{-v}E^{(r)}_{g,2+u_1-u_s}t^v\nonumber\\
&\quad+\hbar \sum_{r=1}^{s}\displaystyle\sum_{v\geq0} \limits\displaystyle\sum_{g=2+u_1-u_s}^{u_1}\limits E^{(r)}_{1+u_1-u_s,g}t^{-v-1}E^{(r)}_{g,2+u_1-u_s}t^{v+1}\nonumber\\
&\quad-\hbar \sum_{r=1}^{s}\displaystyle\sum_{v\geq0} \limits\displaystyle\sum_{x=-q_1}^{-1-q_1+q_r}\limits E^{(r)}_{1+u_1-u_s,g}t^{-v-1}E^{(r)}_{g,2+u_1-u_s}t^{v+1}.\label{eval}
\end{align}
By a direct computation, we obtain
\begin{align}
&\quad\eqref{W-gen}_1+\eqref{eval}_1=-\sum_{1\leq a\leq s}(q_r-q_s+\dfrac{\hbar}{2})E^{(a)}_{u_1-u_s+1,u_1-u_s+2},\label{111}\\
&\quad\eqref{W-gen}_2+\eqref{B}_2+\eqref{B}_4+\eqref{C}_3+\eqref{D}_5\nonumber\\
&=\hbar\displaystyle\sum_{v\geq0} \limits \sum_{r_1<r_2}\limits E_{1+u_1-u_s,2+u_1-u_s}^{(r_1)}t^{-v}E_{1+u_1-u_s,1+u_1-u_s}^{(r_2)}t^{s}\nonumber\\
&\quad+\hbar\sum_{r_1<r_2}\limits\displaystyle\sum_{v\geq0} \limits\sum_{g=u_1-u_s+2}^{u_1-u_{s+1}}\limits E^{(r_2)}_{1+u_1-u_s,g}t^{-v-1} E^{(r_1)}_{g,2+u_1-u_s}t^{v+1},\label{222}\\
&\quad\eqref{W-gen}_3+\eqref{B}_6+\eqref{C}_6+\eqref{D}_3\nonumber\\
&=\hbar\sum_{v\geq0}\sum_{r_1<r_2}\limits\sum_{-q_1+q_{s+1}\leq x<-q_1+q_s}\limits E^{(r_1)}_{x,z_1-u_s+2}t^{-v}E^{(r_2)}_{u_1-u_s+1,x}t^v,\label{333}\\
&\quad\eqref{W-gen}_4+\eqref{B}_5+\eqref{C}_5+\eqref{D}_2\nonumber\\
&=-\hbar\sum_{r_1<r_2}\limits\displaystyle\sum_{v\geq0} \limits\sum_{x=-q_1+q_s-1}^{-q_1+q_{s+1}}\limits E^{(r_1)}_{1+u_1-u_s,x}t^{-v-1} E^{(r_2)}_{x,2+u_1-u_s}t^{v+1},\label{444}\\
&\quad\eqref{W-gen}_5+\eqref{B}_1+\eqref{B}_3+\eqref{C}_2+\eqref{D}_6\nonumber\\
&=\hbar\displaystyle\sum_{v\geq0} \limits \sum_{r_1<r_2}\limits E_{1+u_1-u_s,1+u_1-u_s}^{(r_1)}t^{-v}E_{1+u_1-u_s,2+u_1-u_s}^{(r_2)}t^v\nonumber\\
&\quad+\hbar\displaystyle\sum_{v\geq0} \limits \sum_{r_1<r_2}\limits\displaystyle\sum_{z=2+u_1-u_s}^{u_1-u_{s+1}}\limits E_{1+u_1-u_s,z}^{(r_1)}t^{-v-1}E_{z,2+u_1-u_s}^{(r_2)}t^{v+1},\label{555}\\
&\quad\eqref{W-gen}_6+\eqref{C}_4+\eqref{D}_4+\eqref{eval}_4\nonumber\\
&=-\hbar\sum_{r=1}^{s}\displaystyle\sum_{v\geq0}\limits\sum_{x=-q_1+q_s-1}^{-q_1+q_{s+1}} E^{(r)}_{2+u_1-u_s,x}t^{-v-1}E^{(r)}_{x,1+u_1-u_s}t^{v+1}+\hbar\sum_{r=1}^{s}(q_r-q_s)E^{(r)}_{u_1-u_s+1,u_1-u_s+2}t^{s},\label{666}\\
&\quad\eqref{W-gen}_7+\eqref{C}_1+\eqref{D}_1+\eqref{eval}_2+\eqref{eval}_3\nonumber\\
&=\hbar \sum_{r=1}^{s}\displaystyle\sum_{v\geq0}  \limits E^{(r)}_{1+u_1-u_s,1+u_1-u_s}t^{-v}E^{(r)}_{1+u_1-u_s,2+u_1-u_s}t^v\nonumber\\
&\quad+\hbar \sum_{r=1}^{s}\displaystyle\sum_{v\geq0} \limits\displaystyle\sum_{x=2+u_1-u_s}^{u_1-u_{s+1}}\limits E^{(r)}_{1+u_1-u_s,x}t^{-v-1}E^{(r)}_{x,2+u_1-u_s}t^{v+1}.\label{777}
\end{align}
By a direct computation, we obtain
\begin{align}
&\quad\eqref{111}+\eqref{666}_2=-\dfrac{1}{2}\hbar W^{(1)}_{u_1-u_a+1,q_1-q_a+2},\label{W1}\\
&\quad\eqref{222}_1+\eqref{555}_1+\eqref{777}_1=\hbar\displaystyle\sum_{v\geq0}\limits W^{(1)}_{u_1-u_a+1,u_1-u_a+1}t^{-v}W^{(1)}_{u_1-u_a+1,u_1-u_a+2}t^v,\label{W2}\\
&\quad\eqref{222}_2+\eqref{555}_2+\eqref{777}_2=\hbar\displaystyle\sum_{v\geq 0} \limits\displaystyle\sum_{z=2}^{u_s-u_{s+1}}\limits W^{(1)}_{u_1-u_a+1,u_1-u_a+u}t^{-v-1} W^{(1)}_{u_1-u_a+z,u_1-u_a+2}t^{v+1},\label{W3}\\
&\quad\eqref{333}+\eqref{444}+\eqref{666}_1\nonumber\\
&=-\hbar\displaystyle\sum_{v\geq0}\limits\displaystyle\sum_{z=-q_a+q_{a+1}}^{-1} \limits W^{(1)}_{-q_1+q_a+1,-q_1+q_a+u}t^{-v-1} W^{(1)}_{-q_1+q_a+z,-q_1+q_a+2}t^{v+1}.\label{W4}
\end{align}
By adding \eqref{W1}-\eqref{W4}, we obtain
\begin{align*}
&\quad\bigotimes_{1\leq a\leq s}\ev_{\hbar,\ve_a}^{u_a|q_a,-x_a\hbar}\circ\Delta^{s}(X^+_{1,1})+\hbar W^{(2)}_{u_1-u_s+1,u_1-u_s+2}t\\
&=-\dfrac{1}{2}\hbar W^{(1)}_{u_1-u_s+1,u_1-u_s+2}+\hbar\displaystyle\sum_{v\geq0}\limits W^{(1)}_{u_1-u_s+1,u_1-u_s+1}t^{-v}W^{(1)}_{u_1-u_s+1,u_1-u_s+2}t^v\\
&\quad+\hbar\displaystyle\sum_{v\geq 0} \limits\displaystyle\sum_{z=2}^{u_s-u_{s+1}}\limits W^{(1)}_{u_1-u_s+1,u_1-u_s+z}t^{-v-1} W^{(1)}_{u_1-u_s+z,u_1-u_s+2}t^{v+1}\\
&\quad-\hbar\displaystyle\sum_{v\geq0}\limits\displaystyle\sum_{z=-q_s+q_{s+1}}^{-1} \limits W^{(1)}_{-q_1+q_s+1,-q_1+q_s+z}t^{-v-1} W^{(1)}_{-q_1+q_s+z,-q_1+q_s+2}t^{v+1}.
\end{align*}
Thus, we obtain \eqref{eqqq2}.
\end{proof}
\bibliographystyle{plain}
\bibliography{syuu}
\end{document}